\documentclass[a4paper,12pt,intlimits,oneside,reqno]{amsart}
\usepackage{enumerate}
\usepackage{amsfonts,amsmath, amsxtra,amssymb,latexsym, amscd,amsthm}
\usepackage[dvips]{color}
\usepackage{graphicx}
\usepackage{latexsym,amssymb}
\newtheorem{theorem}{Theorem}[section]

\newtheorem{lemma}[theorem]{Lemma}

\theoremstyle{corollary}
\newtheorem{corollary}[theorem]{Corollary}

\theoremstyle{definition}
\newtheorem{definition}[theorem]{Definition}

\theoremstyle{remark}

\numberwithin{equation}{section}

\newcommand{\comment}[1]{}

\begin{document}

\title [Vector valued maximal Carleson operator]{Vector valued maximal Carleson type operators
on the weighted Lorentz spaces}

\author{Nguyen Minh Chuong}

\address{Institute of mathematics, Vietnamese  Academy of Science and Technology,  Hanoi, Vietnam.}
\email{nmchuong@math.ac.vn}
\thanks{This paper is supported by Vietnam National Foundation for Science and Technology Development (NAFOSTED) under grant number 101.02-2014-51}

\author{Dao Van Duong}
\address{Shool of Mathematics,  Mientrung University of Civil Engineering, Phu Yen, Vietnam.}
\email{daovanduong@muce.edu.vn}

\author{Kieu Huu Dung}
\address{Shool of Mathematics, University of Transport and Communications- Campus in Ho Chi Minh City, Vietnam.}
\email{khdung@utc2.edu.vn}

\begin{abstract}
In this paper, by using the idea of linearizing maximal operators originated by Charles  Fefferman  and the $TT^*$ method of Stein-Wainger,  we  establish a weighted inequality for vector valued maximal Carleson type operators with singular kernels proposed by Andersen and John on the weighted Lorentz spaces with  vector-valued functions.
\end{abstract}

\maketitle

\section{Introduction}\label{section1}
 In 1966, Lennart Carleson \cite{Carleson1966} established the almost everywhere convergence of Fourier series for square-integrable functions by proving the boundedness of weak type $(2,2)$ of the operator,  so-called the Carleson operator, which is defined by
\begin{equation}
\mathcal{C}f(x) =\sup_{\alpha\in\mathbb R}\Big|\int_{-\pi}^{\pi}\frac{e^{-i\alpha y}f(y)}{x-y}dy\Big|.
\end{equation}
In 1968, Richard Hunt \cite{Hunt1967} generalized the Carleson theorem to $L^p[-\pi,\pi]$ spaces for $1<p<\infty$. Next, in 1970 Per Sj\"{o}lin \cite{Sjolin1971} extended the theorem of Carleson to the higher dimensional space by studying the boundedness of the Carleson type operator on $L^p(\mathbb R^n)$ for $1<p<\infty$, which is defined as follows
\begin{equation}
\mathcal{S}f(x) =\sup_{\alpha\in\mathbb R}\Big|\int_{\mathbb R^n}e^{-i\alpha y} K(x-y)f(y)dy\Big|,
\end{equation}
where $K$ is an appropriate Calder\'{o}n-Zygmund kernel in $\mathbb R^n$, that is, it satisfies the following conditions:\\
 {\rm (a) } $K \in {C^{n + 1}}({\mathbb R^n}\backslash \{0\})$;\\
 {\rm (b)} $K(tx) = {t^{ -n}}K(x)$, for all $t > 0$;\\
  {\rm (c)} $\int\limits_{{S^{n - 1}}} {K(x')d\sigma (x')} = 0$.\\
More generally, Elias M. Stein and Stephen Wainger \cite{SteinWainger2001} considered the  $L^p$-boundedness for the maximal Carleson type operators defined as follows. We denote by ${P_\lambda }(x) = \sum\limits_{1 \leq |\alpha | \leq d} {{\lambda _\alpha }}{x^\alpha}$. It is the polynomial in $\mathbb R^n$ of fixed degree $d$ with real coefficients $\lambda:=(\lambda_\alpha)_{1\leq |\alpha|\leq d}$. Denote
\begin{equation}
T_\lambda (f)(x) =\int\limits_{\mathbb R^n} {{e^{i{P_\lambda }(y)}}} K(y)f(x - y)dy.
\end{equation}
The maximal Carleson type operator associated with the family $\{T_\lambda\}$ then  is defined by
\begin{equation}
\mathcal{T^*}(f)(x)= \mathop {\sup }\limits_\lambda  \Big| {T_\lambda (f)(x)}\Big|,
\end{equation}
where the supremum is taken over all the real coefficients $\lambda$ of the polynomial $P_\lambda$. The well-known result given by Stein and Wainger  \cite{SteinWainger2001} is as follows.
\\
\begin{theorem}[Theorem 2 in \cite{SteinWainger2001}]
Assume that  ${P_\lambda }(x) = \sum\limits_{2 \leq |\alpha | \leq d} {{\lambda _\alpha }} {x^\alpha }$ and the kernel $K$ satisfies the following conditions:\\
 {\rm (i)} $K$ is  a tempered distribution and agrees with a $C^1$ function $K(x)$ for $x\not=0$;\\
 {\rm (ii)} $\widehat{K}$ (the Fourier transform of $K$) in $L^\infty(\mathbb R^n)$;\\
 {\rm (iii)} $\left| {\partial_x^\alpha K(x)} \right| \leq A{\left| x \right|^{ - n - \left| \alpha  \right|}}$ for $0\leq |\alpha|\leq 1$. \\
Then the maximal Carleson type operator $\mathcal{T^*}$ is bounded on $L^p(\mathbb R^n)$ for $1<p<\infty$.
\end{theorem}
The theory of weighted norm inequality for the Carleson type operator has been extensively studied by several authors (see, for example, \cite{Hunt1974, Prestini2000, Plinio 2014} and references therein). Recently, Yong Ding and HongHai Liu \cite{DingLiu2011} investigated the boundedness for maximal Carleson type operator $\mathcal{T^*}$ on the weighted Lebesgue with non-smoothness kernels. More precisely, the kernel $K(x)=\frac{\Omega(x)}{|x|^n}$, where $\Omega$ is a measurable function on  $ {\mathbb R^n}\backslash \{0\}$ and satisfies the following properties:%\\[5pt]
 %{\rm (i)} $\Omega$ is a homogeneous function of degree zero;\\[5pt]
 %{\rm (ii)} $\Omega $ is an integrable function on $S^{n-1}$ with zero average;\\[5pt]
 %{\rm (iii)} $\Omega $  satisfies an $L^q$-Dini function ($1<q\leq\infty$), namely, $\int\limits_{0}^{1}\frac{\omega_q(\delta)}{\delta}<\infty$, where $\omega_q(\delta)$ is the integral continuous modulus of $\Omega$ of degree $q$.
%\vskip 5pt
\begin{align}\label{kernelR12}
\Omega\,\text{is a homogeneous function of degree zero};\,\,\,\,\,\,\,\, \,\,\,\,\,\,\,\, \,\,\,\,\,\,\,\,\,\,\,\,\,\,\,\,\,\,\,\,\,\,\,\,\,\,\,\,\,\,\,\,\,\,\,\,\,\,\,\,\,\,\,\,\,\,\,\,\,\,\,\,\,\,\,\,\,\,\,\,\,\,\,\,\,
\\
\Omega\,\text {is an integrable function on $S^{n-1}$ with zero average};\,\,\,\,\,\,\,\,\,\,\,\,\,\,\,\,\,\,\,\,\,\,\,\,\,\,\,\,\,\,\,\,\,\,\,\,\,\,\,\,\,\,\,\,\,\,\,\,\,\,\,\,\,
\end{align}
\begin{equation}\label{kernelR3}
\begin{split}
&\Omega\,\text{satisfies an $L^q$-Dini function ($1<q\leq\infty$)},\,\text{namely, $\int\limits_{0}^{1}\frac{\omega_q(\delta)}{\delta}<\infty$, where}
\\
&\,\text{$\omega_q(\delta)$ is the integral continuous modulus of $\Omega$ of degree $q$}.
\end{split}
\end{equation}
\vskip 5pt
In [11, and references therein], Nguyen Minh Chuong also introduced some other Carleson type operarors and Bi-Carleson operators with well known  interesting estimates.
\vskip 5pt
In 1981, K. Andersen and R. John \cite{Andersen1981} established the weighted norm inequalities for  vector-valued maximal functions and singular integrals on the space $L^p(\ell^r, \omega)$. The class of kernels in this work has the following properties:
\begin{equation}\label{dknhan1}
|K(x)| \leq \frac{A}{{|x{|^n}}},\;\;\big| {\widehat K(x)} \big| \leq A;
\end{equation}

\begin{equation}\label{dknhan2}
\left| {K(x - y) - K(x)} \right| \leq \mu (\left| y \right|/\left| x \right|)|x{|^{- n}}, {\text {for all}} \left| x \right| \geq 2\left| y \right|;
\end{equation}
where $A$ is a constant and $\mu$ is non-decreasing on the positive real half-line, $\mu (2t) \le C\mu (t)$ for all $t > 0$, and satisfies the following Dini condition 
\begin{equation}\label{Dini}
\int\limits_0^1 \frac{\mu (t)}{t}dt < \infty. 
\end{equation}

Note that if $\Omega$ is integrable function on $S^{n-1}$ with zero average, homogeneous of degree zero and satisfies the Dini condition
$$\int\limits_0^1 \frac{\omega (\delta)}{\delta}d\delta< \infty,$$
where 
$$\omega (\delta)=\text{sup}\Big\{\big|\Omega(x)-\Omega(y)\big|: |x|=|y|=1, |x-y|\leq \delta\Big\}, $$
then the Calder\'{o}n-Zygmund kernel $K(x) = \frac{\Omega(x)}{|x|^n}$ belongs to the Andersen- John type kernel as was mentioned above, that is, it satisfies the conditions (\ref{dknhan1})-(\ref{Dini}).
\vskip 5pt
 The goal of this paper is to prove the boundedness of vector-valued maximal Carleson type operators with singular kernels proposed by Andersen and John on the weighted Lorentz spaces by using the idea of linearizing maximal operators due to Charles  Fefferman \cite{Fefferman1973} and the $TT^*$ method of Stein-Wainger given in \cite{SteinWainger2001} (more precisely, Kolmogorov-Seliverstov's stopping-time argument).
%%%%%%%%%%%%%%%%%%%%%%%%%%%
%%%%%%%%%%%%%%%%%%%%%%%%%%%
\section{Vector-valued maximal operators}\label{section2}
%%%%%%%%%%%%%%%%%%%%%%%%%%%
%%%%%%%%%%%%%%%%%%%%%%%%%%%
Before stating our results in this section, let us give some basic facts and notations which will be used throughout this paper. We denote by $\omega(x)$ a weight function, that is a nonnegative locally integrable function on $\mathbb R^n$. Given a measurable set $E$, $\chi_E$ denotes its characteristic function, and $\omega(E)$ denotes the integral $\int\limits_{E}\omega(x)dx$. The letter $C$ denotes a positive constant which is independent of the main parameters, but may be different from line to line. For $f$ a measurable function on $\mathbb R^n$, the distribution function of $f$ associated with the measure $\omega(x)dx$ is defined as follows
$$ d_f (\alpha) = \omega\big( {\left\{ {x \in {\mathbb R^n}: |f(x)| > \alpha} \right\}} \big).
 $$
The decreasing rearrangement of $f$ with respect to the measure $\omega(x)dx$ is the function $f^*$ defined on $[0, \infty)$ by
$$ {f^*}(t) = \text{inf} \big\{ s > 0:{d_f}(s) \leq t\big\}. $$
\begin{definition}[Section 2 in \cite{Hunt1982}] Let $0<p, q\leq\infty$. The weighted Lorentz space $L_\omega^{p,q}(\mathbb R^n)$ is defined as the set of all measurable functions $f$ such that
 $\|f\|_{L^{p,q}(\omega)}<\infty$, where
\[
{\left\| f \right\|_{L^{p,q}(\omega)}} = \left\{ \begin{array}{l}
{\left( {\frac{q}{p}\int\limits_0^\infty  {{{\left[ {{t^{1/p}}f{}^*(t)} \right]}^q}\frac{{dt}}{t}} } \right)^{1/q}}, {\text{ if }} 0 < q < \infty,\\
\mathop \text{sup }\limits_{t > 0} {t^{1/p}}{f^*}(t),\,\,\,\,\,\,\,\,\,\,\,\,\,\,\,\,\,\,\,\,\,\,\,\,\,\,\,\,{\text{ if }}q = \infty.
\end{array} \right.
\]
\end{definition}
For simplicity, instead of $\left\| f \right\|_{L^{p,q}(\omega)}$, we use $\left\| f \right\|_{pq}$. It is useful to remark that when $p=q$, then $L^{p,p}_\omega(\mathbb R^n)$ is just the usual weighted Lebesgue space. For more details about the weighted Lorentz space as well as its applications, we refer the interested readers to the works \cite{{Hunt1966}, {Hunt1982}, {MJCarro2007}, {Garcia-Cuerva}}. 
\vskip 5pt
Let $\vec{f}=\{f_k\}$ be a sequence of measurable functions on $\mathbb R^n$. We denote
\[
|\vec{f}(x)|_r = {\left( {\sum\limits_{k = 1}^\infty  {|{f_k}(x){|^r}} } \right)^{1/r}}.
\]
As usual, the vector-valued weighted Lorentz space $L_\omega^{p,q}(\ell^r, \mathbb R^n)$ 
is defined as the set of all sequences of measurable functions $\vec{f}=\{f_k\}$ such that
\[
\big\| \vec{f}\big\|_{pq}=\big\| \vec{f} \big\|_{L^{pq}(\ell^r,  \omega)} = \big\| {|\vec{f}(x){|_r}} \big\|_{pq} < \infty.
\]

We denote by $S$ the linear space of sequences $\vec{f}=\{f_k\}$  such that each $f_k(x)$ is a simple function on $\mathbb R^n$ and $f_k (x)\equiv 0$ for all sufficiently large $k$. It is interesting to remark that $S$ is dense in $L_\omega^{p, q}(\ell^r, \mathbb R^n)$ for all $1\leq p, q, r <\infty$, see \cite{{Blasco}, {Gretsky}}.
\vskip 5pt
 Next, we present some basic facts on the class of weight functions $A(p, q)$. Let us be either $1<p<\infty$ and $1\leq q\leq \infty$ or $p = q = 1$. The weight function $\omega(x)$ is in $A(p, q)$ if there exists a positive constant $C$ such that for any cube $Q$, we have
\[
{\left\| {{\chi _Q}} \right\|_{pq}}{\left\| {{\chi _Q}{\omega^{ - 1}}} \right\|_{p'q'}} \le C|Q|.
\]
Note that in the particular case $p=q$, we have $A(p, p)=A_p$, the class of Muckenhoupt weighted functions \cite{Muckenhoupt72}. Also, it is proved in \cite{{Hunt1982}} that when $1<p<\infty$ and $1<q\leq\infty$, then $A(p, q)=A_p$. Next, we recall several   important results related to the class of weight functions $A(p, q)$, which are used in the sequel.
\begin{lemma}[Lemma 2.7 in \cite{Hunt1982}] \label{Huntbde27} 
Let $\omega(x)\in A(p, q)$. Then, $\omega\in A(r,s)$ if either $r =p$ and $ 1\leq s\leq q$ or  $r >p$ and $ 1\leq s\leq\infty$.
\end{lemma}

\begin{lemma}[Lemma 4.4 in \cite{Hunt1982}]\label{Huntbde44}
Let $1<p<\infty, 1<q\leq\infty$, and $\omega\in A(p,q)$. Then, there exist two real numbers $r, s$ greater than 1 with  $r<p$ and $\omega\in A(r,s)$.
\end{lemma}
%%%%%%%%%%%%%%%%%%%%%%%%%%%%%%%%%%%%%%%%%%%%%%%%%%%%%%%%%%%%
%%%%%%%%%%%%%%%%%%%%%%%%%%%%%%%%%%%%%%%%%%%%%%%%%%%%%%%%%%%%
\begin{lemma}[Lemma 2.8 in \cite{Hunt1982}]\label{Huntbde2.8}
The weighted function $\omega\in A(p,1)$ if and only if there exists a positive constant $C$ such that for any cube, $Q$, and subset $E\subset Q$,
\[
\frac{|E|}{|Q|}\leq C \left( \frac{\omega(E)}{\omega(Q)}\right)^{1/p}.
\]
\end{lemma}
\begin{lemma}[Corollary 9.2.6 in \cite{GrafakosModern}]\label{Grafakos9.2.6}
If $1<p<\infty$ then $A_p = \bigcup\limits_{q\in (1, p)}A_q$.
\end{lemma}
\begin{lemma}[Lemma 2.5 in \cite{Hunt1982}]\label{Huntbde2.5}
If $1\leq q\leq p <\infty$ and $\big\{ E_j\big\}_{j\geq 1}$ is a collection of sets such that $\sum\limits_{j\geq 1}{\chi_{E_j}(x)}\leq C$, then 
\[
\sum\limits_{j \geq 1}\big\| \chi_{E_j}f\big\|_{L^{p,q}(\omega)}^{p}\leq C\big\| f\big\|_{L^{p,q}(\omega)}^{p}.
\]
\end{lemma}
%%%%%%%%%%%%%%%%%%%%%%%%%%%%%%%%%%%%%%%%%%%%%%%%%%%%%%
%%%%%%%%%%%%%%%%%%%%%%%%%%%%%%%%%%%%%%%%%%%%%%%%%%%%%%
Now, let us mention the important Marcinkiewicz interpolation type result related to the Lorentz spaces with vector-valued functions. For further information, the interested readers may refer to \cite{{Hunt1966}, {MJCarro2007}} for the scalar-valued case and to \cite{{Andersen1981}, {Bergh}, {Garcia-Cuerva}} for the case of vector-valued functions.
\begin{theorem}\label{interpolation1}
Suppose $T$ is a sublinear operator satisfying 
$$ { \big\| {T(\vec{f})} \big\|_{p_i^{'}q_i^{'}}}\leq {C_i}{\big\| \vec{f} \big\|_{{p_i}{q_i}}},\,\,i = 0,1, $$
with ${p_0} < {p_1},p_0^{'} \ne p_1^{'}$. Then there is a positive constant $C_\theta$ such that
\[
{\big\| {T(\vec{f})} \big\|_{p_\theta ^{'}s}} \leq {C_\theta }{\big\| \vec{f} \big\|_{{p_\theta }q}},
\]
with $q\leq s$, $0<\theta <1$, and 
$\left(\dfrac{1}{p_\theta}, \dfrac{1}{p_\theta^{'}}   \right)=(1-\theta)\left( \dfrac{1}{p_0},\dfrac{1}{p_0^{'}}\right)+\theta\left(\dfrac{1}{p_1}, \dfrac{1}{p_1^{'}}\right)$.
\end{theorem}
From Theorem 3.11 and Theorem 3.12 in \cite{Garcia-Cuerva}, we also have the Riesz-Thorin interpolation type results related to the Lorentz spaces with vector-valued functions. 

\begin{theorem}\label{interpolation2}
If $T$ is a linear operator  from $L^{p_0, q_0}(\ell^r,\omega) + L^{p_1, q_1}(\ell^r,\omega)$ to $L^{p_0^ {'},q_0^{'}}(\ell^r,\omega) + L^{p_1^{'}, q_1^{'}}(\ell^r,\omega)$  and 
$${ \big \| {T(\vec{f})} \big\|_{p_i^{'}q_i^{'}}}\le {C_i}{\big\| \vec{f} \big\|_{{p_i}{q_i}}},\,\,i = 0,1,$$
then we have
\[
{\big\| {T(\vec{f})} \big\|_{p_{\theta} ^{'}q_{\theta} ^{'}}} \leq C_0^{1 - \theta }C_1^\theta {\big\| {\vec{f}} \big\|_{{p_\theta }{q_\theta }}},
\]
where $0<\theta<1$, $\left(\dfrac{1}{p_\theta}, \dfrac{1}{p_\theta^{'}}, \dfrac{1}{q_\theta}, \dfrac{1}{q_\theta^{'}}   \right)=(1-\theta)\left( \dfrac{1}{p_0},\dfrac{1}{p_0^{'}}, \dfrac{1}{q_0},\dfrac{1}{q_0^{'}}\right)+
\theta\left(\dfrac{1}{p_1}, \dfrac{1}{p_1^{'}}, \dfrac{1}{q_1}, \dfrac{1}{q_1^{'}}\right)$ and ${q_\theta}^{'}=q_\theta$.
\end{theorem}

\begin{theorem}\label{interpolation2'}
Suppose $T$ is a linear operator  from $L^{p_0, 1}(\ell^r,\omega) + L^{p_1, 1}(\ell^r,\omega)$ to $L^{p_0^ {'}, \infty}(\ell^r,\omega) + L^{p_1^{'}, \infty}(\ell^r,\omega)$  satisfying
$${ \big \| {T(\vec{f})} \big\|_{p_i\infty}}\le {C_i}{\big\| \vec{f} \big\|_{{p_i}{1}}},\,\,i = 0,1,$$
with $p_0\neq p_1$. Then we have
\[
{\big\| {T(\vec{f})} \big\|_{p_{\theta} r}} \leq C_0^{1 - \theta }C_1^\theta {\big\| {\vec{f}} \big\|_{{p_\theta }{r}}},
\]
where $0<\theta<1$, $1\leq r\leq \infty$  and $\dfrac{1}{p_\theta} = \dfrac{1-\theta}{p_0}+ \dfrac{\theta}{p_1}$.
\end{theorem}
Let us recall that the Hardy-Littlewood maximal function is defined by
$$Mf(x)=\sup_{Q}\frac{1}{|Q|}\int\limits_{Q}|f(y)|dy, $$
where the supremum is taken over all cubes $Q$ of Lebesgue measure $|Q|$, centered at $x$ with sides parallel to the coordinate axis. Denote $M(\vec{f})$ by $\{Mf_k\}$ for $\vec{f}=\{f_k\}$. By the definition of Lorentz spaces and Theorem 3.1 in \cite{Andersen1981} due to Andersen and John, the following lemma, which actually extends some matters in \cite{Hunt1982} to the case of vector-valued functions, is easily given. The proof is trivial and is left to the reader.
\begin{lemma}\label{lemma26}
Let $1<r<\infty$, $1<q\leq p<\infty$, and $\omega\in A(p, q)$. We then have
$$ \big\|M(\vec{f})\|_{L^{p,\infty}(\ell^r, \omega)}\leq C\big\|\vec{f}\big\|_{L^{p,q}(\ell^r, \omega)}, $$
for all $\vec{f}\in L^{p,q}(\ell^r, \omega)$.
\end{lemma}
Applying Lemma \ref{lemma26} and Theorem \ref{interpolation1}, we have the following result.

\begin{theorem}\label{theorem27}
Let $1<p, q, r<\infty$ and $\omega\in A(p,q)$. Then, for every $\vec{f}\in L^{p,q}(\ell^r, \omega)$, 
$$  \big\|M(\vec{f})\|_{L^{p,q}(\ell^r, \omega)}\leq C\big\|\vec{f}\big\|_{L^{p,q}(\ell^r, \omega)}.$$
\end{theorem}

\begin{proof} 
The proof of the theorem is not difficult, but for convenience to the reader, we briefly give here.
Indeed,  by Lemma \ref{Huntbde44}, there exists a real number $p_1$ for $1<p_1<p$ so that
$\omega\in A(p_1, q)$. With the notation $q_1=\min\{p_1, q\}$, we also have $\omega\in A(p_1, q_1)$ by Lemma \ref{Huntbde27}. Similarly, choosing $p_2>p$, we also have $\omega\in A(p_2, q_1)$. Now, using Lemma \ref{lemma26} and Theorem \ref{interpolation1} we immediately obtain the desired result. 
\end{proof}
%%%%%%%%%%%%%%%%%%%%%%%%%%%%%%%%%%%%%%%%%%%%%%%%%%%%%%%%%%%%%%%%%%%%%%%%%%
%%%%%%%%%%%%%%%%%%%%%%%%%%%%%%%%%%%%%%%%%%%%%%%%%%%%%%%%%%%%%%%%%%%%%%%%%%
We also extend and research the object of the work \cite{Hunt1982} to vector-valued maximal functions. It seems to be  difficult to work for the class of weights which are different from the Muckenhoupt weights. Applying some results of the work \cite{Hunt1982} and several techniques of K. Andersen and R. John \cite{Andersen1981} (more precisely, due to C. Fefferman and E. M. Stein \cite{FeffermanStein}), we obtain the following result, which extends and strengthens some  interesting results due to H. M. Chung et al.  in \cite{Hunt1982}.
\begin{theorem}\label{dinhly1}
 If $1<p, r<\infty$ and ${\big\| {M\big(\vec{f}\big)} \big\|_{L^{p,\infty }(\ell^r,\omega)}} \leq C {\big\|\vec{f} \big\|_{L^{p,1}(\ell^r, \omega)}}$, for all $\vec{f} \in L^{p,1}({\ell^r,\omega})$ then $\omega\in A(p,1)$. 
Conversely, there are two cases as  follows: 
\begin{enumerate}
\item [\rm (i)] Let $1<p, r <\infty$, $\omega\in \bigcup\limits_{q \in (1,p)} {A(q,1)}$. Then, for every $\vec{f} \in L^{p, 1}({\ell^r,\omega})$,
\[
 {\big\| {M\big(\vec{f}\big)} \big\|_{L^{p,\infty}(\ell^r, \omega)}} \leq C {\big\|\vec{f} \big\|_{L^{p, 1}(\ell^r,\omega)}}.
 \]
\item [\rm (ii)] Let $1<p< r<\infty$, $\omega\in A(p,1)$. Then, for every $\vec{f} \in L^{p,1}({\ell^r,\omega})$,
\[
 {\big\| {M\big(\vec{f}\big)} \big\|_{L^{p,\infty}(\ell^r, \omega)}} \leq C {\big\|\vec{f} \big\|_{L^{p,1}(\ell^r,\omega)}}.
 \]
\end{enumerate}
\end{theorem}
\begin{proof}
In order to prove the necessary condition, it is sufficient to choose $\vec{f}= ( f, 0,..., 0, ...)$. Then, by Theorem 1 in \cite{Hunt1982}, it is immediately shown that $\omega\in A(p,1)$.
Next, we will prove the sufficient conditions of the theorem. 
\\
$\rm (i)$ Using Lemma \ref{Huntbde27} and Lemma \ref{Grafakos9.2.6}, it is clear that 
 \[
 \bigcup\limits_{q\in (1,p)}A(q,1)= A_p.
 \]
Thus, by the result of Theorem 3.1 in \cite{Andersen1981} and  the property of Lorentz norms, 
we immediately obtain
\[
 {\big\| {M\big(\vec{f}\big)} \big\|_{L^{p,\infty}(\ell^r, \omega)}} \lesssim {\big\|\vec{f} \big\|_{L^{p,1}(\ell^r,\omega)}}, \text{ for all  }\vec{f}\in L^{p,1}(\ell^r,\omega).
 \]
$\rm (ii)$ As usual, we can assume without loss of generality that $\vec{f}\in S$. For $\alpha >0$, from 
the Calder\'{o}n-Zygmund decomposition \cite{book1}, there exists a sequence of $\big\{Q_j\big\}$, whose interiors are disjoint such that
\begin{equation}\label{Andersen3.5}
\big|\vec{f}(x)\big|_{r}\leq \alpha,\,\, x\notin \Omega =\bigcup\limits_{j=1}^{\infty} Q_j;
\end{equation}
\begin{equation}\label{Andersen3.6}
\alpha \leq \frac{1}{|Q_j|}\int\limits_{Q_j}{\big|\vec{f}(x)\big|_{r}dx}\leq 2^n\alpha, \,\text{for all}\, j\in \mathbb Z^{+}.
\end{equation}
Note that  $\vec{f}= \vec{f'}+\vec{f''}$, where $\vec{f'}= \big\{ f_{k}^{'}\big\}, f_{k}^{'}(x)= f_k(x)\chi_{\mathbb R^n\setminus\Omega}(x)$. Hence, it is easy to see that
\[
\big|M\big(\vec{f}\big)(x)\big|_{r}\leq \big|M\big(\vec{f'}\big)(x)\big|_{r}+\big|M\big(\vec{f''}\big)(x)\big|_{r}.
\]
To obtain the desired result, it is sufficient to show that
\begin{equation}\label{Andersen3.7}
\omega\left(\left\{ x\in\mathbb R^n: \big|M\big(\vec{f'}\big)(x)\big|_{r}>\alpha \right\}\right)\leq C_{r,p}.{\alpha}^{-p}\big\| \vec{f}\big\|_{L^{p,1}(\ell^r,\omega)}^p,
\end{equation}
and 
\begin{equation}\label{Andersen3.8}
\omega\left(\left\{ x\in\mathbb R^n: \big|M\big(\vec{f''}\big)(x)\big|_{r}>\alpha \right\}\right)\leq C_{r,p}.{\alpha}^{-p}\big\| \vec{f}\big\|_{L^{p,1}(\ell^r,\omega)}^p.
\end{equation}
By assuming that $\omega\in A(p,1)$ and $r>p$, and applying Lemma \ref{Huntbde27}, we get $\omega\in A_r$. Thus, by Theorem 3.1 in \cite{Andersen1981}, we obtain
\begin{equation}\label{Andersen3.4}
\omega\big(\big\{ x\in\mathbb R^n: \big|M\big(\vec{f'}\big)(x)\big|_{r}>\alpha \big\}\big)\leq C_{r}.{\alpha}^{-r}\big\| \vec{f'}\big\|_{L^{r}(\ell^r,\omega)}^r.
\end{equation}
Hence, by (\ref{Andersen3.5}), the property of Lorentz norms and the inequality $\big|\vec{f'}(x)\big|_{r}^r\leq \alpha^{r-p}|\vec{f}(x)|_{r}^p$ , we easily imply that the inequality $(\ref{Andersen3.7})$ holds.
\\
To estimate the inequality (\ref{Andersen3.8}), we need to define $\overline{f}=\big\{ \overline{f}_k\big\}$ by
\[
\overline{f}_k(x)= \left\{ \begin{array}{l}
\frac{1}{|Q_j|}\int\limits_{Q_j}{|f_k(y)|dy},\,\,\, x\in Q_j, j = 1,2, ...,
\\
0,\,\,\,\,\,\,\,\,\,\,\,\,\,\,\,\,\,\,\,\,\,\,\,\,\,\,\,\,\,\,\,\,\,\,\,\,\,\, \text{otherwise}.
\end{array} \right.
\]
Here, we denote that $\overline{\Omega}=\bigcup\limits_{j=1}^{\infty}\overline{Q}_j$ and $\overline{Q}_j$ is the cube with the same center as $Q_j$ but with $\text{diameter}\,(\overline{Q}_j)= 2n.\text{diameter}\,(Q_j)$. Next, we have to estimate the following inequality
\begin{equation}\label{OverOmega}
\omega\big(\overline{\Omega}\big)\leq C.\alpha^{-p}\big\| \vec{f}\big\|_{L^{p,1}(\ell^r, \omega)}^p.
\end{equation}
By $\omega\in A(p,1)$ and using Lemma \ref{Huntbde2.8}, we have
\[
\frac{|Q_j|}{|\overline{Q}_j|}\leq \left(\frac{\omega(Q_j)}{\omega(\overline{Q}_j)}\right)^{1/p},\,\text{for all}\, j \in\mathbb Z^{+}.
\]
Therefore, $\omega\big(\overline{Q}_j\big)\leq C\omega\big(Q_j\big)$, for all $j\in\mathbb Z^{+}$. Thus, using the inequality  (\ref{Andersen3.6}), H\"{o}lder's inequality in Lorentz space and the definition of $A(p,1)$, we obtain
\begin{align}
\omega\big(\overline{\Omega}\big)&\leq C.\sum\limits_{j\geq 1}\omega\big(Q_j\big)\leq C.\alpha^{-p}\sum\limits_{j\geq 1}\omega\big(Q_j\big)\left(\frac{1}{|Q_j|}\int\limits_{Q_j}{\big|\vec{f}(x)\big|_{r}\omega^{-1}\omega dx} \right)^{p}\nonumber
\\
&\leq C.\alpha^{-p} \sum\limits_{j\geq 1}\omega\big(Q_j\big)\big|Q_j\big|^{-p}\big\| \chi_{Q_j}\big|\vec{f}\big|_{r}\big\|_{L^{p,1}(\omega)}^{p}.\big\|\chi_{Q_j}w^{-1}\big\|_{L^{p',\infty}(\omega)}^p\nonumber
\\
&\leq C.\alpha^{-p}\sum\limits_{j\geq 1}\big\|\chi_{Q_j}\big|\vec{f}\big|_{r}\big\|_{L^{p,1}(\omega)}^{p}.\nonumber
\end{align}
On the other hand, since the family of cubes $\big\{Q_j\big\}$ are disjoint, by Lemma \ref{Huntbde2.5}, it implies that 
\[
\sum\limits_{j\geq 1}\big\|\chi_{Q_j}\big|\vec{f}\big|_{r}\big\|_{L^{p,1}(\omega)}^{p}\leq \big\|\vec{f}\big\|_{L^{p,1}(\ell^r,\omega)}^{p},
\]
 which completes the proof of the inequality (\ref{OverOmega}). As a consequence, we have
\begin{equation}\label{Omega}
\omega\big(\Omega\big)\leq C.\alpha^{-p}\big\| \vec{f}\big\|_{L^{p,1}(\ell^r,\omega)}^p.
\end{equation}
 Now, we consider the sequence $\overline{f}$. We also obtain in a similar  argument way
 to the proof of (\ref{Andersen3.4}) that 
\begin{equation}\label{AndersenS3.11}
\omega\Big(\big\{x\in \mathbb R^n: \big|M\big(\overline{f}\big)(x)\big|_{r}>\alpha \big\}\Big)\leq C_{r}.{\alpha}^{-r}\int\limits_{\mathbb R^n}{\big|\overline{f}(x)\big|_{r}^r\omega(x)dx}.
\end{equation}
From the definition of $\overline{f}$, it is clear that $\text{supp}\big(|\overline{f}|_r\big)\subset \Omega$, and using (\ref{Andersen3.6}), we get $\big|\overline{f}(x)\big|_{r}\leq 2^n\alpha$. Therefore, by (\ref{Omega}) and (\ref{AndersenS3.11}), it follows that 
 \begin{equation}\label{AndersenS3.11'}
\omega\Big(\big\{x\in \mathbb R^n: \big|M\big(\overline{f}\big)(x)\big|_{r}>\alpha \big\}\Big)\leq C.\int\limits_{\Omega}{\omega(x)}dx\leq C.{\alpha}^{-p}\big\|\vec{f}\big\|_{L^{p,1}(\ell^r,\omega)}^{p}.
\end{equation}
It is well known that Theorem 1 in \cite{FeffermanStein}, we have
\[
M\big(f_k^{''}\big)(x)\leq c. M\big(\overline{f}_k\big)(x), \,\,x\notin\overline{\Omega},
\]
and by (\ref{OverOmega}), (\ref{AndersenS3.11'}), we thus obtain the following inequality
\begin{align}
\omega\left(\left\{ x\in\mathbb R^n: \big|M\big(\vec{f''}\big)(x)\big|_{r}>\alpha \right\}\right)&\leq\omega\big(\overline{\Omega}\big)+ \omega\left(\left\{ x\notin\overline{\Omega}: \big|M\big(\vec{f''}\big)(x)\big|_{r} > {\alpha} \right\}\right)\nonumber
\\
&\leq C.{\alpha}^{-p}\big\|\vec{f}\big\|_{L^{p,1}(\ell^r,\omega)}^{p},\nonumber
\end{align}
 which completes the proof for the inequality (\ref{Andersen3.8}). Finally, since $S$ is dense in $L^{p,1}(\ell^r,\omega)$, the proof of the theorem is finished.
\end{proof}

%%%%%%%%%%%%%%%%%%%%%%%%%%%%%%%%%%%%%%%%%%%%%%%%%%%%%%%%%%%%%%%%%%%%%%%%%%
%%%%%%%%%%%%%%%%%%%%%%%%%%%%%%%%%%%%%%%%%%%%%%%%%%%%%%%%%%%%%%%%%%%%%%%%%%
Let $\{K_k(x)\}$ denote a sequence of singular convolution kernels satisfying the above conditions (\ref{dknhan1})-(\ref{Dini}) with a uniform constant $A$ and a fixed function $\mu$ not dependent of $k$. We define the singular integral operator $T_k$ and maximal singular integral operator $T_k^*$, respectively, as follows

$$ T_k(f)(x) = p.v.\int\limits_{{\mathbb R^n}} {K_k(y)f(x - y)dy}, $$
$$ T_k^*f(x) = \mathop {\sup }\limits_{\varepsilon  > 0} \Big| {\int\limits_{|x - y| > \varepsilon } {K_k(y)f(x - y)dy} } \Big|.  $$
For $\vec{f}=\{f_k\}$, define $ T( {\vec{f}}) = \left\{ {{T_k}\left( {{f_k}} \right)} \right\}$ and ${T^*}(\vec{f}) = \left\{ {T^*_k\left( {{f_k}} \right)} \right\}$. Now, we will give the vector-valued weighted norm inequalites for $T$ and $T^*$ on the weighted Lorentz spaces, which generalise some well-known results in \cite{Andersen1981}.
\\
\begin{theorem}\label{theorem28}
Let $1 < p, q, r < \infty$ and $\omega \in {A(p, q)}$. Then, for all $\vec{f}\in S$, we have
$$
\big\| {{T^ * }\big( {\vec{f}} \big)} \big\|_{{L^{p,q}}({\ell^r},\omega)} \leq C {\big\| {\vec{f}} \big\|_{{L^{p,q}}({\ell^r},\omega)}}.
$$
\end{theorem}
\begin{proof} 
 By strong arguments in the same way as Theorem \ref{theorem27} together with using Theorem 5.2 in \cite{Andersen1981} and Theorem \ref{interpolation1}, we obtain the desired result.
\end{proof}
%%%%%%%%%%%%%%%%%%%%%%%%%%%%
%%%%%%%%%%%%%%%%%%%%%%%%%%%%
%From Theorem \ref{dinhly1} and Lemma 5.1 in \cite{Andersen1981}, we obtain endpoint estimate for the vector-valued %maximal singular integral operator on the weighted Lorentz spaces.
\begin{theorem}\label{theorem28'}
Let $1 < p< r < \infty$, and $\omega \in {A(p, 1)}$. Then, for all $\vec{f}\in S$, we have
$$
\big\| {{T^ * }\big( {\vec{f}} \big)} \big\|_{{L^{p,\infty}}({\ell^r},\omega)} \leq C {\big\| {\vec{f}} \big\|_{{L^{p,1}}({\ell^r},\omega)}}.
$$
\end{theorem}
\begin{proof}
By Lemma 5.1 in \cite{Andersen1981}, there are two constants $C_r, \delta >0$ such that
\begin{equation}\label{eq21}
{d_{{{\big| {{T^ * }( {\vec{f} })} \big|}_{{r}}}}}(2\alpha )
\le {C_r}{\gamma ^\delta }{d_{{{\big| {{T^ * }( {\vec{f} })} \big|}_{{r}}}}}(\alpha ) + {d_{M( {{{| {\vec{f} }|}_{{r}}}})}}(\gamma \alpha ) + {d_{{{\big| {M( {\vec{f}})} \big|}_{{r}}}}}(\gamma \alpha ), 
\end{equation}
for all $\alpha,\gamma >0$. The inequality (\ref{eq21}) allows us to obtain
\begin{align}
{\big\| {{{\big| {{T^ * }\big( {\vec{f}} \big)} \big|}_{{r}}}} \big\|_{{L^{p,\infty }}(\omega)}} &\leq 2{C_r}^{1/p}{\gamma ^{\delta /p}}{\big\| {{{\big| {{T^ * }\big( {\vec{f}} \big)} \big|}_{{r}}}} \big\|_{{L^{p,\infty }}(\omega)}}\nonumber
\\
&+ \frac{2}{\gamma }{\big\| {M\big( {{{| {\vec{f}}|}_{{r}}}} \big)} \big\|_{_{{L^{p,\infty }}(\omega)}}} + \frac{2}{\gamma }{\big\| {{{\big| {M\big( {\vec{f}} \big)} \big|}_{{r}}}} \big\|_{_{{L^{p,\infty }}(\omega)}}}.\nonumber
\end{align}
Now, choose $\gamma=\gamma_0$, dependent of $p, r$, satisfying  $1 - 2{C_r}^{1/p}{\gamma _0}^{\delta /p} \geq \frac{1}{2}$. We  then get
\[
{\big\| {{{\big| {{T^ * }\big( {\vec{f} } \big)} \big|}_{{r}}}} \big\|_{{L^{p,\infty }}(\omega)}}\leq \frac{4}{{{\gamma _0}}}{\big\| {M\big( {{{\big| {\vec{f} } \big|}_{{r}}}} \big)} \big\|_{_{{L^{p,\infty }}(\omega)}}} + \frac{4}{{{\gamma _0}}}{\big\| {{{\big| {M\big( {\vec{f} } \big)} \big|}_{{r}}}} \big\|_{_{{L^{p,\infty }}(\omega)}}}.
\]
Thus, by Theorem \ref{dinhly1}, the proof is completed.
\end{proof}
Obviously,  Theorem \ref{theorem28} and Theorem \ref{theorem28'} also allow us to obtain the following useful results.

\begin{corollary}\label{corollary29}

 Let $1 < p, q, r < \infty$ and  $ \omega \in {A(p, q)}$. We then get
$$
{\big\| {{T}\big( {\vec{f} }\big)} \big\|_{{L^{p,q}}({\ell^r}, \omega)}} \leq C {\big\| {\vec{f}} \big\|_{{L^{p,q}}({\ell^r}, \omega)}},  $$
for all $\vec{f}\in S $.
\end{corollary}
\begin{corollary}\label{corollary29'}
If $1 < p< r < \infty$ and $\omega \in {A(p, 1)}$, then we have
$$
{\big\| {{T}\big( {\vec{f} }\big)} \big\|_{{L^{p,\infty}}({\ell^r}, \omega)}} \leq C {\big\| {\vec{f}} \big\|_{{L^{p,1}}({\ell^r}, \omega)}},  $$
for all $\vec{f}\in S $.
\end{corollary}
\section{Vector-valued maximal Carleson type operator}\label{section3}
%%%%%%%%%%%%%%%%%%%%%%%%%%%
%%%%%%%%%%%%%%%%%%%%%%%%%%%
In this section, we will discuss the boundedness of vector-valued maximal Carleson type operator on the weighted Lorentz spaces. We also consider the sequence of convolution kernels $\{K_k(x)\}$ as in Section 2 above. Let us denote by ${P_\lambda }(x) = \sum\limits_{2 \leq |\alpha | \leq d} {{\lambda _\alpha }} {x^\alpha }$ the polynomial in $\mathbb R^n$ of fixed degree $d$ (no linear terms) with real coefficients $\lambda=(\lambda_\alpha)_{2\leq |\alpha|\leq d}$. The vector-valued maximal Carleson type operator is defined by 
\begin{equation}
 {\mathcal{T^*}}\big( {\vec{f}} \big) = \left\{ {\mathcal{T}_k^*({f_k})} \right\}_{k = 1}^\infty,
\end{equation}
with 
\begin{equation}
\mathcal{T}_k^*({f_k})(x) = \mathop {\sup }\limits_\lambda  \left| {{\mathcal{T} _{\lambda, k}}({f_k})(x)} \right| = \mathop {\sup }\limits_\lambda  \Big| {\int\limits_{{\mathbb R^n}} {{e^{i{P_\lambda }(y)}}} {K_k}(y){f_k}(x - y)dy} \Big|,
\end{equation}
where the supremum is taken over all the real coefficients $\lambda$ of the polynomial $P_\lambda$. Our main results in this paper are the following.
\\
\begin{theorem}\label{Theorem31}
 Let $1 < p, q, r < \infty $ and $\omega\in{A_p}$. Then, we have 
\begin{equation}
{\big\| {{\mathcal{T^*}  }\big( {\vec{f}} \big)} \big\|_{{L^{p,q}}({\ell^r}, \omega)}} \leq C {\big\| {\vec{f}} \big\|_{{L^{p,q}}({\ell^r,} \omega)}},
\end{equation}
for all $\vec{f}\in L^{p, q}(\ell^r,\omega)$.
\end{theorem}
We are also interested in the scalar-valued maximal Carleson type operators on the weighted Lorentz space $L^{p,1}(\omega)$.
\begin{theorem}\label{Theorem31'}
 Let $1<p <\infty$. Suppose that $\omega\in A(p,1)$ and there exists a constant $\varepsilon>0$ satisfying $\omega^{1+\varepsilon}\in A_p$. Then, we have 
\begin{equation}
{\big\| {{\mathcal{T^*}  }\big( {{f}} \big)} \big\|_{{L^{p,\infty}}(\omega)}} \leq C {\big\| {{f}} \big\|_{{L^{p,1}}(\omega)}},
\end{equation}
for all ${f}\in L^{p,1}(\omega)$.
\end{theorem}
The idea for the proof of Theorem \ref{Theorem31} and Theorem \ref{Theorem31'} mainly follows the arguments of Stein and Wainger in \cite{SteinWainger2001} (see also in \cite{DingLiu2011}), namely, it is based on the Kolmogorov-Seliverstov stopping-time argument as well as some van der Corput estimates for oscillatory integrals. However, the class of singular convolution kernels considered in this section is relatively general and somewhat different from the kernels studied in \cite{SteinWainger2001}, \cite{DingLiu2011}. We also remark that the Stein-Weiss theorem on $L^p$ interpolation with change of measure can not be extended to the Lorentz spaces, see in \cite{Ferreyra}.
Thus, we need to give some new techniques for our arguments. 

Before proving Theorem \ref{Theorem31} and Theorem \ref{Theorem31'}, for the sake of the reader, we want to recall some well-known results due to Stein and Wainger in \cite{SteinWainger2001} and due to Stein and Weiss in \cite{SteinWeiss}. 
\begin{lemma}[Proposition 2.1 and Proposition 2.2 in \cite{SteinWainger2001}]\label{mdeVDC}
Suppose that $\varphi$ is a $C^1$ function defined in the unit ball  ${U} = \left\{ {x \in {\mathbb R^n}:\left| x \right| \leq 1} \right\}$.  Let $P(x)=\sum_{1\leq\alpha\leq d}\lambda_\alpha x^\alpha $, and let $V$ be any convex subet of $U$. Then, the following statements are true.\\
 {\rm (i) } There exists a positive constant $C$ independent of $P, \varphi, V$ such that
$$ \Big|{\int\limits_V  {{e^{iP(x)}}} \varphi (x)dx} \Big| \leq C{\left| \lambda  \right|^{ - 1/d}}\mathop {\sup }\limits_{x \in U} \left( {\left| {\varphi (x)} \right| + \left| {\nabla \varphi (x)} \right|} \right). $$
{\rm (ii) }There exists a positive constant $C$ independent of $P$ such that
$$\left| {\left\{ {x \in U:\left| {P(x)} \right| \leq \varepsilon } \right\}} \right| \leq C{\varepsilon ^{1/d}}{\left| \lambda  \right|^{ - 1/d}},
{\text { for all }} \varepsilon  > 0.$$
\end{lemma}
Let us denote $B_{33/16}=\{x\in\mathbb R^n: |x|\leq 33/16\}$. For any subset $E$ of $B_{33/16}$, we write ${\left( {{\chi _E}} \right)_a}(x) = {a^{ - n}}{\chi _E}(x/a)$. Given a positive real number $\varepsilon $, the maximal function $M_\varepsilon $ is defined as follows
$$ M_{\epsilon}(f)(x) = \mathop {\sup }\limits_{|E| \leq {\epsilon}\hfill\atop
\scriptstyle{\rm{           }}a > 0\hfill} \left| f \right|*{\left( {{\chi _E}} \right)_a}(x), $$
where the supremum is taken over all subsets $E$ of $B_{33/16}$ of measure less than $\varepsilon$ and all $a >0$.
%%%%%%%%%%%%%%%%%%%%%%%%%%%%%%%%%%%%%%%%%%%%%%%%%%%%%%%%%%%%%
%%%%%%%%%%%%%%%%%%%%%%%%%%%%%%%%%%%%%%%%%%%%%%%%%%%%%%%%%%%%%
\begin{lemma}[Proposition 3.1 in \cite{SteinWainger2001}]\label{Steinmde31}
There exists a positive constant $C$ independent of $\varepsilon$ such that for all $f \in {L^2(\mathbb R^n)}$, 
$${\left\| {{M_{\epsilon}}(f)} \right\|_{{L^2}}} \le c{\epsilon}^{1/2}{\left\| f \right\|_{{L^2}}}.$$
\end{lemma}
Next, let us recall the Stein-Weiss interpolation theorem with change of measure.
\begin{theorem}[Theorem 2.11 in \cite{SteinWeiss}]\label{noisuyStein}
Let $1<p_0, p_1<\infty$ and $u_0, v_0, u_1, v_1$ be weighted functions. Suppose that the sublinear operator $T$ satisfies 
${\left\| {T(f)} \right\|_{{L^{{p_i}}}({u_i})}} \leq C_{i}{\left\| f \right\|_{{L^{{p_i}}}({v_i})}},$ for $i = 0,1$. Then, there exists a constant  $C\in (0, C_0^\theta C_1^{1 - \theta })$ such that
\[
{\left\| {T(f)} \right\|_{{L^{{p_\theta }}}({u_\theta })}} \leq C {\left\| f \right\|_{{L^{{p_\theta }}}({v_\theta })}},
\]
where $ 1/p_\theta =\theta/p_0 + (1 - \theta )/p_1$,  
${u_\theta } = u_0^{({p_\theta }/{p_0})\theta }u_1^{({p_\theta }/{p_1})(1 - \theta )}$, 
${v_\theta } = v_0^{({p_\theta }/{p_0})\theta }v_1^{({p_\theta }/{p_1})(1 - \theta )}$, for any $0<\theta<1$.
\end{theorem}
%%%%%%%%%%%%%%%%%%%%%%%%%%%%%%%
%%%%%%%%%%%%%%%%%%%%%%%%%%%%%%
As a consequence of Theorem \ref{noisuyStein}, we also have the analogous result for the vector - valued case as follows.
\begin{corollary}\label{HquaSteinWeiss}
Let $1<p_0, p_1, r_0, r_1 <\infty$ and $u_0, v_0, u_1, v_1$ be weighted functions. Suppose that T is a sublinear operator satisfying
$${ \big \| {T(\vec{f})} \big\|_{L^{p_i}(\ell^{r_i}, u_i)}}\le {C_i}{\big\| \vec{f} \big\|_{L^{p_i}(\ell^{r_i}, v_i)}},\,\,i = 0,1.$$
Then, we have
\[
{\big\| {T(\vec{f})} \big\|_{L^{p_\theta}(\ell^{p_\theta}, u_\theta)}} \leq C_0^{1 - \theta }C_1^\theta {\big\| {\vec{f}} \big\|_{L^{p_\theta}(\ell^{p_\theta}, v_\theta)}},
\]
where $ 1/p_\theta =\theta/p_0 + (1 - \theta )/p_1$,  
${u_\theta } = u_0^{({p_\theta }/{p_0})\theta }u_1^{({p_\theta }/{p_1})(1 - \theta )}$, 
${v_\theta } = v_0^{({p_\theta }/{p_0})\theta }v_1^{({p_\theta }/{p_1})(1 - \theta )}$, for any $0<\theta<1$.
\end{corollary}
%%%%%%%%%%%%%%%%%%%%%%%%%%%%%%%%%%%%%%%%%
%%%%%%%%%%%%%%%%%%%%%%%%%%%%%%%%%%%%%%%%%
Now, we are in a position to give the proof of our main result. Firstly, we will solve Theorem \ref{Theorem31} for the case $p=q$.
\\
\begin{theorem}\label{lemma31}
 Let $1 < p, r < \infty $ and $\omega\in{A_p}$. Then, we have 
\begin{equation}
{\big\| {{\mathcal{T^*}  }\big( {\vec{f}} \big)} \big\|_{{L^{p}}({\ell^r}, \omega)}} \leq C {\big\| {\vec{f}} \big\|_{{L^{p}}({\ell^r,} \omega)}},
\end{equation}
for all $\vec{f}\in L^{p}(\ell^r,\omega)$.
\end{theorem}
%%%%%%%%%%%%%%%%%%%%%%%%%%%
%%%%%%%%%%%%%%%%%%%%%%%%%%%
\begin{proof}[The proof of Theorem \ref{lemma31}]
We can assume without loss of generality that $\vec{f}\in S$. Next, it is a simple matter to see that 
\begin{equation}\label{eq34}
 \mathcal{T} _k^*({f_k})(x) = \mathop {\sup }\limits_\lambda  \left| {{\mathcal{T} _{\lambda ,k}}({f_k})(x)} \right| \leq \mathop {\sup }\limits_{\lambda  \ne 0} \left| {{\mathcal{T} _{\lambda ,k}}({f_k})(x)} \right| + \left| {{T_k}({f_k})(x)} \right|.
\end{equation}
For a simple function $f_k$ and $x\in\mathbb R^n$, there exists a sequence of measurable stopping-time functions  $\lambda (x,k) = \left\{ {{\lambda _\alpha }(x,k)} \right\}$ satisfying
\begin{equation}\label{eq35}
 \left| {{ \mathcal{T} _{\lambda (x,k),k}}({f_k})(x)} \right| \geq \frac{1}{2}\mathop {\sup }\limits_{\lambda  \ne 0} \left| {{ \mathcal{T} _{\lambda ,k}}({f_k})(x)} \right|. 
\end{equation}
For convenience, we set 
 $${{ \mathcal{T} _{\lambda (x)}}({\vec{f}})(x)} =\lbrace {{\mathcal{T} _{\lambda (x,k),k}}({f_k})(x)} \rbrace_{k=1}^{\infty}.$$
From (\ref{eq34}), (\ref{eq35}) and Corollary \ref{corollary29}, in order to prove the theorem, it is sufficient to show that there is a positive constant $C$, not dependent on $\lambda (x,k)$,  such that
\begin{equation}\label{eq36}
\big\|{ \mathcal{T} _{\lambda (\cdot)}}\big({\vec{f}}\big)\big\|_{{L^{p}}({\ell^r}, \omega)}\leq C 
\big\|{\vec{f}}\big\|_{{L^{p}}({\ell^r}, \omega)}.
\end{equation}
In what follows, we follow some notations used in \cite{DingLiu2011, SteinWainger2001}.
As usual, we  choose a nonnegative bump function $\psi  \in C_0^\infty ({\mathbb R^n})$ such that 
$\text{supp}(\psi)  \subseteq \left\{ {y \in {\mathbb R^n}:1/4 < \left| y \right| \leq 1} \right\}$  and $\sum\limits_{j =  - \infty }^\infty  {\psi _j}(y) = 1$, for all  $ y \ne 0$, where ${\psi _j}(y) = \psi ({2^{ - j}}y)$. We write  $N(\lambda) = \sum\limits_{2 \le |\alpha | \le d} {{{\left| {{\lambda _\alpha }} \right|}^{\frac{1}{{|\alpha |}}}}} $ and ${\psi _{j,\lambda }}(y) = {\psi _j}\left( {N\left( \lambda  \right)y} \right)$. Then, the kernels $K_k$ are decomposed as follows
\begin{equation}\label{eq37}
\begin{split}
K_k(y)&= \sum\limits_{j =  - \infty }^0 {{\psi _{j,\lambda }}(y){K_k}(y)} + \sum\limits_{j = 1}^\infty  {{\psi _{j,\lambda }}(y){K_k}(y)}
\\
&={K_{0,k}}(y) + \sum\limits_{j = 1}^\infty  {{K_{j,k}}} (y).
\end{split}
\end{equation}
Denote
$$ {{\mathcal{T}^0 _{\lambda (x,k),k}}({f_k})(x)}= {\int\limits_{{\mathbb R^n}} {{e^{i{P_{\lambda (x,k)}}(y)}}} {K_{0,k}}(y){f_k}(x - y)dy}, $$
$$ {{\mathcal{T}^j _{\lambda (x,k),k}}({f_k})(x)}= {\int\limits_{{\mathbb R^n}} {{e^{i{P_{\lambda (x,k)}}(y)}}} {K_{j,k}}(y){f_k}(x - y)dy},$$
 $${{ \mathcal{T}^0 _{\lambda (x)}}({\vec{f}})(x)} =\lbrace {{\mathcal{T}^0 _{\lambda (x,k),k}}({f_k})(x)} \rbrace_{k=1}^{\infty},$$
 $${{ \mathcal{T}^j _{\lambda (x)}}({\vec{f}})(x)} =\lbrace {{\mathcal{T}^j _{\lambda (x,k),k}}({f_k})(x)} \rbrace_{k=1}^{\infty}.$$
By (\ref{eq37}), we have
\begin{equation}\label{eq38} 
{\big\| {{\rm{ }}{{{{\mathcal{T} _{\lambda (\cdot)}}\big(\vec{f}\big)} }}} \big\|_{{L^{p}}(\ell^r, \omega)}}
 \lesssim {\big\| {{\rm{ }}{{{{\mathcal{T}^0 _{\lambda (\cdot)}}\big(\vec{f}\big)} }}} \big\|_{{L^{p}}(\ell^r, \omega)}} + \sum\limits_{j = 1}^\infty {\big\| {{\rm{ }}{{{{\mathcal{T}^j _{\lambda (\cdot)}}\big(\vec{f}\big)} }}} \big\|_{{L^{p}}(\ell^r, \omega)}}. 
\end{equation} 
Here, we write $a \lesssim b$ to mean that there is a positive constant $C$, independent of the main parameters, such that $a\leq Cb$.  The next arguments are divided into the following several steps.\\[5pt]
\\
$\bullet$ \textit{Step 1:} The estimate of $\mathcal{T}^0 _{\lambda (\cdot)}.$\\
By a similar argument as in \cite{DingLiu2011}, we have
\begin{align}
\left| {{\mathcal{T}^0}_{\lambda (x,k),k}({f_k})(x)} \right|&\leq  \Big| {\int_{\left| y \right| \leq \frac{1}{{2N(\lambda (x,k))}}} {{e^{i{P_{\lambda (x,k)}}(y)}}{K_k}(y){f_k}(x - y)} dy} \Big| \nonumber
\\
&+\int\limits_{\frac{1}{{2N(\lambda (x,k))}} \le \left| y \right| \le \frac{1}{{N(\lambda (x,k))}}} {\left| {{K_k}(y)} \right|\left| {{f_k}(x - y)} \right|} dy\nonumber\\
&:=I_1+I_2.\nonumber
\end{align}
From condition (\ref{dknhan1}) of the kernels $K_k$, it is not difficult to show that
$$ I_1\lesssim M({f_k})(x)+ \big|{T_k}({f_k})(x)\big| + T_k^*({f_k})(x), $$
and $$ {I_2}\lesssim M({f_k})(x). $$
%%%%%%%%%%%%%%%%%%%%%%%%%%%%
%%%%%%%%%%%%%%%%%%%%%%%%%%%%
Applying the boundedness of the vector-valued maximal functions and maximal singular integrals on the weighted Lorentz  spaces in Section \ref{section2}, there is a positive constant $C$ independent of $\lambda(\cdot)$ such that
\begin{equation}\label{tau0}
{\big\| {{ \mathcal{T}^0 _{\lambda (\cdot)}}\big(\vec{f}\big)} \big\|_{{L^{p}}(\ell^r, \omega)}}
\leq C{\big\| \vec{f}\big\|_{{L^{p}}(\ell^r, \omega)}},
\end{equation} 
holds under the given conditions of Theorem \ref{lemma31}.\\[5pt]
\\
%%%%%%%%%%%%%%%%%%%%%%%%%%%%%%%%%%%%%%%%%%%%%%%%%%%%%%%%%%%%%%%%%%%%%%%%%%%%%%%%%%%%%%%%%%%%%%%%%%%%%%%%%%%%%%%%
%%%%%%%%%%%%%%%%%%%%%%%%%%%%%%%%%%%%%%%%%%%%%%%%%%%%%%%%%%%%%%%%%%%%%%%%%%%%%%%%%%%%%%%%%%%%%%%%%%%%%%%%%%%%%%%%
$\bullet$ \textit{Step 2:} The estimate of $\mathcal{T}^j _{\lambda (\cdot)}.$\\
We take another nonnegative bump function  $\phi  \in C_0^\infty ({\mathbb{R}^n})$ such that $\big\|\phi\big\|_{L^1} =1$ and $\text{supp} (\phi ) \subseteq \big\{y\in \mathbb{R}^n: {\left| y \right| \le {2^{ - 5}}} \big\}$. Let us 
denote ${\phi _a}(x) = {a^{ - n}}\phi (x/a)$, for all $a > 0$. For some $\sigma  > 0$ small enough, which will be taken later, we let $a_1 = \frac{2^{j(1 - \sigma)}}{N(\lambda (x,k))}$, and define  
\[
{L_{j,\lambda (x,k)}}(y) = {K_{j,k}} * {\phi _{a_{1}}}(y),
\]
and
\[
{R_{j,\lambda (x,k)}}(y) = {K_{j,k}}(y) - {L_{j,\lambda (x,k)}}(y).
\]
%%%%%%%%%%%%%%%%%%%%%%%%%%%%%%%%%%%%%%%%%%%%%%%%%%%%%%%%%
From the definition of $L_{j,\lambda(x,k)}$, we have ${L_{j,\lambda (x,k)}} \in C_0^\infty ({\mathbb{R}^n})$. We now estimate the support of the function $L_{j,\lambda(x,k)}$. Notice first that 
\[
\text{supp} \left( {{L_{j,\lambda (x,k)}}} \right) \subset \overline { \text{supp}\left( {\phi _{a_1}} \right) + \text{supp} \left( {{K_{j,k}}} \right)}.
\]
Let $u \in\text{supp}(\phi_{a_1})\subset\big\{ {| y| \le {2^{ - 5}}a_1} \big\}$ and $v\in \text{supp}(K_{j,k})\subset \big\{ {\frac{{{2^{j - 2}}}}{{N(\lambda (x,k))}} \le \left| y \right| \le \frac{{{2^j}}}{{N(\lambda (x,k))}}} \big\}$. We have
\[
\left\| {u + v} \right\| \le \left\| u \right\| + \left\| v \right\| \le \frac{{{2^j}(1 + {2^{ - 5}})}}{{N(\lambda (x,k))}},
\]
and
\[
\left\| {u + v} \right\| \ge \left\| v \right\| - \left\| u \right\|\ge \frac{{{2^{j - 5}}({2^3} - 1)}}{{N(\lambda (x,k))}}.
\]
From the above estimates, we can obtain the following interesting inequality, which is actually better than one given in \cite{DingLiu2011},
\[
{\mathop{\text{supp}}\nolimits} \left( {{L_{j,\lambda (x,k)}}} \right) \subset \left\{y\in\mathbb R^n: {\frac{{{{7.2}^{j}}}}{{32N(\lambda (x,k))}} \le \left| y \right| \leq \frac{{{{33.2}^{j}}}}{{32N(\lambda (x,k))}}} \right\}.
\]
Hence, we get 
\[
{\mathop{\text{supp}}\nolimits} \left( {{R_{j,\lambda (x,k)}}} \right) \subset \left\{y\in\mathbb R^n: {\frac{{{{7.2}^{j}}}}{{32N(\lambda (x,k))}} \le \left| y \right| \leq \frac{{{{33.2}^{j}}}}{{32N(\lambda (x,k))}}} \right\}.
\]
We also define two useful vector-valued operators $\mathfrak{T}_{\lambda (\cdot)}^j$ and $\mathfrak{R}_{\lambda (\cdot)}^j$ as follows
$$\mathfrak{T} _{\lambda (x)}^j\big({\vec{f}}\big)(x)=\big\{{\mathfrak{T}}_{\lambda (x,k)}^j({f_k})(x)\big\}_{k=1}^{\infty}\;\; {\text{and }}\;\; \mathfrak{R}_{\lambda (x)}^j\big({\vec{f}}\big)(x)=\big\{\mathfrak{R}_{\lambda (x,k)}^j({f_k})(x)\big\}_{k=1}^{\infty},$$
 where 
\[
\mathfrak{T}_{\lambda (x,k)}^{j}({f_k})(x) = \int\limits_{{\mathbb R^n}} {{e^{i{P_{\lambda(x,k)}}(y)}}{L_{j,\lambda (x,k)}}(y){f_k}(x - y)dy},\] 
\[
\mathfrak{R}_{\lambda (x,k)}^j({f_k})(x) = \int\limits_{{\mathbb R^n}} {{e^{i{P_{\lambda(x,k)}}(y)}}{R_{j,\lambda (x,k)}}(y){f_k}(x - y)dy}.
\]
From the decomposition of kernels $K_{j,k}$, it follows that
\begin{equation}\label{paper5'}
{\big| {{\mathcal{T} }^j_{\lambda (x)}(\vec{f})(x)} \big|_{r}} \lesssim {\big| {\mathfrak{R}_{\lambda (x)}^j(\vec{f})(x)} \big|_{r}}+{\big| {\mathfrak{T} _{\lambda (x)}^j(\vec{f})(x)} \big|_{{r}}}.
\end{equation}
\\
$\bullet$ \textit{Step 2.1:} The estimate of $\mathfrak{R}^j _{\lambda (\cdot)}.$\\
By a trivial calculation, we have
\begin{align}
\big| {{R_{j,\lambda (x,k)}}(y)} \big|&\leq \int\limits_{{\mathbb R^n}} {\left| {{K_k}(y)} \right|}  {\left| {{\psi _{j,\lambda }}(y) - {\psi _{j,\lambda }}(y - z)} \right|} \left| {\phi _{a_1}(z)} \right|dz\nonumber
\\
&+\int\limits_{{\mathbb R^n}} {\left| {{\psi _{j,\lambda }}(y - z)} \right|} {\left| {{K_k}(y) - {K_k}(y - z)} \right|} \left| {\phi _{a_1}(z)} \right|dz\nonumber
\\
 &={J_1} + {J_2}.\nonumber
\end{align}
Using the mean value theorem and the property of kernel (\ref{dknhan1}), we obtain that
\begin{align}
{J_1} &\leq  \frac{{A{{\left\| {\nabla \psi } \right\|}_{{L^\infty }}}}}{{{{\left| y \right|}^n}}}\left( {{2^{ - j}}N(\lambda (x,k))} \right)\int\limits_{{\mathbb R^n}} {\left| z \right|\left| {{\phi _{a_1}}(z)} \right|dz}\nonumber\\
&\leq A. {2^{ - 5(n + 1)}}.| {{B_n}}|.{\left\| {\nabla \psi } \right\|_{{L^\infty }}}.{\left\| \phi  \right\|_{{L^\infty }}}. {2^{ - j\sigma }}\frac{1}{{{{\left| y \right|}^n}}}\nonumber\\
& \lesssim {2^{ - j\sigma }}\frac{1}{|y|^n},\nonumber
\end{align}
where $|B_n|$ denotes the Lebesgue measure of the unit ball in $\mathbb R^n$.
\\
Next, we observe that $\text{supp} \left( {{\phi _{a_1}}} \right) \subseteq \big\{y\in \mathbb{R}^n: {|y| \le {2^{ - 5}}a_1} \big\}$. Therefore, we have the control
\[
{J_2}\lesssim {\left( {a_1} \right)^{ - n}}\int\limits_{|z| \le {2^{ - 5}}a_1} {\big| {{K_k}(y) - {K_k}(y - z)} \big|} dz.
\]
Take $y \in {\mathop{\text{supp}}\nolimits} \left( {{R_{j,\lambda (x,k)}}} \right)$. For $\left| z \right| \le {2^{ - 5}}a_1$, it can easily show that $\left| y \right| > 2\left| z \right|$. Then, following the property of kernel (\ref{dknhan2}), we have
\[
\big|{{K_k}(y) - {K_k}(y - z)}\big| \le \mu \big(\left| z \right|/\left| y \right|\big)|y|^{- n}.
\]
Hence,
\[
{J_2} \lesssim \left( {a_1} \right)^{ - n}\frac{1}{{{{\left| y \right|}^n}}}\int\limits_{|z| \le {2^{ - 5}}a_1} {\mu \left( \left| z \right|/\left| y \right| \right)} dz.
\]
Since $\left| z \right|/\left| y \right| \leq {2^{ - j\sigma - 2}}$ and $\mu(t)$ is non-decreasing, we obtain
${J_2} \lesssim \mu \big( {{2^{ - j\sigma  - 2}}} \big)\frac{1}{| y |^n}$. 
Therefore,
\[
\big| {{R_{j,\lambda (x,k)}}(y)} \big| \lesssim \Big({{2^{ - j\sigma }} + \mu\big(2^{- j\sigma - 2}\big)} \Big) \frac{1}{|y| ^n}.
\]
By defining of the operator $\mathfrak{R}_{\lambda (\cdot,k)}^j({f_k})$, we can show that 
\begin{align}\label{paper191}
\big| {{\mathfrak{R}^j_{\lambda (x,k)}}(f_k)(x)}\big| &\leq \int\limits_{\frac{{{{7.2}^j}}}{{32.N\left( {\lambda (x,k)} \right)}} \le \left| y \right| \le \frac{{{{33.2}^j}}}{{32.N\left( {\lambda (x,k)} \right)}}} {\left| {{R_{j,\lambda (x,k)}}(y)} \right|.\left| {{f_k}(x - y)} \right|dy}\nonumber \\
&\lesssim \Big({{2^{ - j\sigma }} + \mu ({2^{ - j\sigma  - 2}})} \Big)\int\limits_{\frac{{{{7.2}^j}}}{{32.N\left( {\lambda (x,k)} \right)}} \le \left| y \right| \le \frac{{{{33.2}^j}}}{{32.N\left( {\lambda (x,k)} \right)}}} {\frac{{\left| {{f_k}(x - y)} \right|}}{{{{\left| y \right|}^n}}}dy}\nonumber \\
&\lesssim \Big( {{2^{ - j\sigma }} + \mu \big(2^{- j\sigma - 2}\big)}\Big) M(f_k)(x).
\end{align}
From (\ref{paper191}) and Theorem \ref{theorem27} together with assuming $\omega\in A_p$, we obtain that
\begin{equation}\label{paper21}
{\big\| {\mathfrak{R}^j_{\lambda (\cdot)}\big(\vec{f} \big)} \big\|_{{L^{p}}(\ell^r, \omega)}} \lesssim \Big( {{2^{ - j\sigma }} + \mu \big(2^{ - j\sigma - 2}\big)} \Big){\big\| {{\vec{f} }}\big\|_{{L^{p}}(\ell^r,\omega)}}.
\end{equation}
\\
%%%%%%%%%%%%%%%%%%%%%%%%%%%%%%%%%%%%%%%%%%%%%%%%%%%%%%%%%%%%%%%%%%%%%%%%%%%%%%%%%%%%%%%%%%%%%%%%
%%%%%%%%%%%%%%%%%%%%%%%%%%%%%%%%%%%%%%%%%%%%%%%%%%%%%%%%%%%%%%%%%%%%%%%%%%%%%%%%%%%%%%%%%%%%%%%%
$\bullet$ \textit{Step 2.2:} The estimate of $\mathfrak{T}^j _{\lambda (\cdot)}.$\\
Although there is not the assumption of the homogeneous kernel as in \cite{DingLiu2011}, we may also give the boundedness of $L_{j,\lambda(x,k)}(y)$ here, that is, 
\begin{equation}\label{paper1}
\begin{split}
\big| {{L_{j,\lambda (x,k)}}\left( y \right)} \big| &\leq A.{2^{jn\sigma }}.{\left[ {{2^{ - j}}N(\lambda (x,k))} \right]^n}.{\vartheta_j}\big( {{2^{ - j}}N(\lambda (x,k))y} \big)
\\
&\lesssim {2^{jn\sigma }}{\left[ {{2^{ - j}}N(\lambda (x,k))} \right]^n},
\end{split}
\end{equation}
where ${\vartheta_j}(y) = \int\limits_{{\mathbb R^n}} {\dfrac{|\psi(y - u)|}{{|y - u|^n}}|\phi ({2^{j\sigma }}u)} |du$. Indeed, we have 
\begin{align}
& {L_{j,\lambda (x,k)}}\left( {\frac{{{2^j}y}}{{N(\lambda (x,k))}}} \right) = \int\limits_{{\mathbb R^n}} {{K_{j,k}}} \left( {\frac{{{2^j}y}}{{N(\lambda (x,k))}} - z} \right){\phi _{a_1}}(z)dz\nonumber\\
&= {a_1^{ - n}}\int\limits_{{\mathbb R^n}} {\psi \left( {{2^{ - j}}N(\lambda (x,k))\left({\frac{{{2^j}y}}{{N(\lambda (x,k))}} - z} \right) } \right){K_k}} \left( {\frac{{{2^j}y}}{{N(\lambda (x,k))}} - z} \right)\phi \left(\frac{z}{a_1}\right)dz\nonumber\\
&={ a_1^{ - n}}\int\limits_{{\mathbb R^n}} {\psi \left( {y - {2^{ - j}}N(\lambda (x,k))z} \right){K_k}} \left( {\frac{{{2^j}}}{{N(\lambda (x,k))}}\left( {y - {2^{ - j}}N(\lambda (x,k))z} \right)} \right)\phi \left(\frac{z}{a_1}\right)dz.\nonumber
\end{align}
Set $u = {2^{ - j}}N(\lambda (x,k))z$. It follows that  ${z}/{a_1} =  {2^{j\sigma }}u.$ Therefore, 
\begin{align}\label{eq3.14}
&\Big| {{L_{j,\lambda (x,k)}}\left( {\frac{{{2^j}y}}{{N(\lambda (x,k))}}} \right)} \Big|\nonumber
\\ 
&\leq {a_1^{ - n}}{[{2^j}/N(\lambda (x,k))]^n}\int\limits_{{\mathbb R^n}} {\Big| {\psi (y - u){K_{k}}\left( {\frac{{{2^j}}}{{N(\lambda (x,k))}}(y - u) } \right)\phi ({2^{j\sigma }}u)} \Big|} du\nonumber
\\
 &\leq {a_1^{ - n}}{[{2^j}/N(\lambda (x,k))]^n}\frac{A}{{{{[{2^j}/N(\lambda (x,k))]}^n}}}\int\limits_{{\mathbb R^n}} {\frac{|\psi (y - u)|}{{|y - u{|^n}}}|\phi ({2^{j\sigma }}u)} |du\nonumber
\\
 &=A.2^{j\sigma n}{\left[{2^{ - j}}.N(\lambda (x,k))\right]^n}\int\limits_{{\mathbb R^n}} {\frac{ |\psi (y - u)| }{{|y - u{|^n}}}|\phi ({2^{j\sigma }}u)} |du.
\end{align}
From the inequality (\ref{eq3.14}), to prove the inequality (\ref{paper1}), it is sufficient to show that  ${\vartheta_j}(2^{-j}N(\lambda (x,k))y) $ is upper bounded. But, this is not difficult, and its proof is left to the reader. We also have the following estimates
\begin{align}
\left| {\mathfrak{T}_{\lambda (x,k)}^j({f_k})(x)} \right| & \leq \int\limits_{\frac{{{{7.2}^{j - 5}}}}{{N\left( {\lambda (x,k)} \right)}} \le |y| \le \frac{{{{33.2}^{j - 5}}}}{{N\left( {\lambda (x,k)} \right)}}} {\left| {{L_{j,\lambda (x,k)}}(y)} \right|\left| {{f_k}(x - y)} \right|} dy\nonumber\\
&\leq C.{2^{jn\sigma }}{\left[ {{2^{ - j}}N(\lambda (x,k))} \right]^n}.\int\limits_{\frac{{{{7.2}^{j - 5}}}}{{N\left( {\lambda (x,k)} \right)}} \le |y| \le \frac{{{{33.2}^{j - 5}}}}{{N\left( {\lambda (x,k)} \right)}}} {\left| {{f_k}(x - y)} \right|} dy\nonumber\\
&\leq {C}{.2^{jn\sigma }}.M({f_k})(x).\nonumber
\end{align}
Consequently, we obtain that
\begin{equation}\label{paper*}
{\big| {\mathfrak{T}^j_{\lambda (x)}\big(\vec{f}\big)(x)} \big|_{{r}}} \lesssim {2^{jn\sigma }}{\big| {M\big(\vec{f} \big)(x)} \big|_{{r}}}.
\end{equation}
\\
%%%%%%%%%%%%%%%%%%%%%%%%%%%%%%%%%%%%%%%%%%%%%%%%%%%%%%%%%%%%%%%%%%%%%%
$\bullet$ \textit{Step 2.2.1:} The estimate of ${\kappa _{j,r,k}}(x,z)$ (see (\ref{kappa}) below).
\\ 
For $j \in \mathbb{N}$, we denote ${A_{j,\lambda }} \circ \lambda  = {\left( {{{\left( {\frac{{{2^j}}}{{N\left( \lambda  \right)}}} \right)}^{\left| \alpha  \right|}}{\lambda _\alpha }} \right)_{2 \le \left| \alpha  \right| \le d}}$. For $k\in \mathbb{N}, r >0$, we set 
\[
{U_{j,r,k}} = \big\{ {x \in {\mathbb R^n}:r \le \big| {{A_{j,\lambda (x,k)}}\circ \lambda (x,k)} \big| < 2r} \big\},
\]
and
\[
{\Phi^\lambda }(y) = {e^{i{P_\lambda }(y)}}{L_{j,\lambda }}(y).
\]
For simplicity of notation, we denote ${\widetilde \Phi ^\lambda }(y) = \overline {{\Phi}}^\lambda ( - y) = {e^{ - i{P_\lambda }( - y)}}\overline {{L_{j,\lambda }}} ( - y)$. Thus, the operator $\mathfrak{T}^{j,r,k}$ defined by ${\mathfrak{T}^{j,r,k}}(f)(x) = \mathfrak{T} _{\lambda (x,k)}^j(f)(x){\chi _{{U_{j,r,k}}}}(x) $ is of the form as follows
\[
{\mathfrak{T}^{j,r,k}}(f)(x) = \left( {\int\limits_{{\mathbb R^n}} {{\Phi ^{\lambda (x,k)}}} (y)f(x - y)dy} \right){\chi _{{U_{j,r,k}}}}(x).
\]
We quite look for the adjoint operator of ${\mathfrak{T}^{j,r,k}}$,  denoted by ${\left( {{\mathfrak{T}^{j,r,k}}} \right)^ * }$, satisfying
\[
\left( {\mathfrak{T}^{j,r,k}} \right){\left( {\mathfrak{T} ^{j,r,k}} \right)^*}(f)(x) = \int\limits_{{\mathbb R^n}} {{\kappa _{j,r,k}}(x,z)f(z)} dz,
\]
where 
\begin{equation}\label{kappa}
\kappa _{j,r,k}(x,z)={\Phi ^{\lambda (x,k)}} * {\widetilde \Phi ^{\lambda (z,k)}}(x - z).\chi_{U_{j,r,k}}(x).\chi_{U_{j,r,k}}(z).
\end{equation}
We choose $x, z \in {U_{j,r,k}}$ and take $h = \frac{N(\lambda (z,k))}{N(\lambda (x,k))}$. Below we will give the estimate of $\kappa _{j,r,k}(x,z)$ by considering the following two cases.
\\
\\
\textit{$\circ$ Case 1}: Assume that  $h\leq 1$. For $P_{\lambda (x,k)}(y)= P_{{A_{j,\lambda (x,k)}\circ\lambda(x,k)}}\big(2^{ - j}N(\lambda (x,k))y\big)$, we have
\begin{align}
&{\Phi ^{\lambda (x,k)}} * {\widetilde \Phi ^{\lambda (z,k)}}\big({2^j}u/N(\lambda (z,k))\big)={\big( {{2^{ - j}}N(\lambda (z,k))} \big)^n}\times
\nonumber
\\
&\,\,\times\int\limits_{{\mathbb R^n}} {{e^{i{P_{{A_{j,\lambda (x,k)}\circ\lambda(x,k)}}}(y) - i{P_{{A_{j,\lambda (z,k)}\circ\lambda(z,k)}}}( - u + hy)}}} {\widetilde L_{j,\lambda (x,k)}}(y){\overline{\widetilde L}_{j,\lambda (z,k)}}( - u + hy)dy,\nonumber
\end{align}
where ${\widetilde L_{j,\lambda (\cdot,k)}}(y) = {\big( {{2^{ - j}}N(\lambda (\cdot,k))}\big)^{ - n}}{L_{j,\lambda (\cdot,k)}}\big( {{2^j}y/N(\lambda (\cdot,k))} \big)$. By the information of $L_{j,\lambda}$, we get
\[
{\widetilde L_{j,\lambda (\cdot,k)}} \in C_0^\infty (\mathbb{R}^n)\,\,\,{ \rm{and}}\,\,\, {\text{supp}}\left( {{{\widetilde L}_{j,\lambda (\cdot,k)}}} \right)\subseteq \left\{ {\frac{7}{{32}} \leq \big| y \big| \leq \frac{{33}}{{32}}} \right\}.
\] 
\\
We only need to give $|u|\le \frac{{33}}{{16}}$ and define  
\[
G(y) = {\widetilde L_{j,\lambda (x,k)}}(y){\overline{\widetilde L}}_{j,\lambda (z,k)}(- u + hy),\,\text{for all}\, y \in {\mathbb R^n}.
\]
Therefore, we have $G \in C_0^\infty (\mathbb{R}^n)$ and estimate as follows              
\begin{equation}\label{paper2}
\left| {\nabla G(y)} \right| + \left| {G(y)} \right| \lesssim {2^{2jn\sigma  + j\sigma }} + {2^{2jn\sigma }}, \,\text{for all } y \in {\mathbb R^n}.
\end{equation}
Indeed, from (\ref{paper1}), we get $\big| {{{\widetilde L}_{j,\lambda (x,k)}}(y)} \big| \lesssim {2^{jn\sigma }} $ and $\big| \overline{{\widetilde L}}_{j,\lambda (z,k)}(-u + hy) \big| \lesssim {2^{jn\sigma }}$. Thus, we obtain
\begin{equation}\label{paper2'}
|G(y)| \lesssim {2^{2jn\sigma }},\,\text{for all }\,y \in {\mathbb R^n}. 
\end{equation}
By a trivial calculation, we have
\begin{equation}\label{napGy}
\begin{split}
G_{y_i}(y)= \big( {{{\widetilde L}_{j,\lambda (x,k)}}} \big)_{{y_i}}(y){\overline{{\widetilde L}}_{j,\lambda (z,k)}}(-u + hy) 
\,\,\,\,\,\,\,\,\,\,\,\,\,\,\,\,\,\,\,\,\,\,\,\,\,\,\,\,\,\,\,\,\,\,\,\,\,\,\,\,\,\,\,\,\,\,\,
\\
\,\,\,\,\,\,\,\,\,\,\,\,\,\,\,\,\,\,\,\,\,\,\,\,\,\,\,\,\,\,\,\,\,\,\,\,\,\,\,\,\,\,\,\,\,\,\,\,\,\,\,\,\,\,\,\,\,\,\,\,\,\,\,\,\,\,\,\,\,\,\,\,\,\,\,\,\,\,\,\,\,\,\,\,\,\,\,\,\,\,\,+ {{\widetilde L}_{j,\lambda (x,k)}}(y){{\big( {{\overline{{\widetilde L}}_{j,\lambda (z,k)}}} \big)}_{{t_i}}}(-u + hy)h.
\end{split}
\end{equation}
Next, we get
\[
{\big( {{{\widetilde L}_{j,\lambda (x,k)}}} \big)_{{y_i}}}(y)={a_1^{- 1}}.{\big( {{2^{ - j}}N(\lambda (x,k))} \big)^{-n - 1}}\big( {{K_{j,k}}*{{\left( {{\phi _{{t_i}}}} \right)}_{a_1}}} \big)\big( 2^jy/N(\lambda (x,k)) \big).
\]
By a similar manner way to the proof of the inequality (\ref{paper1}), we have
\[
\big| {\big( {{K_{j,k}}*{{\left( {{\phi _{{t_i}}}} \right)}_{a_1}}} \big)(y)} \big| \lesssim {2^{jn\sigma }}{\left[ {{2^{ - j}}N(\lambda (x,k))} \right]^n}, \,\text{for all } y \in {\mathbb R^n}.
\]
Therefore
\begin{equation}\label{napGy1}
\big| {{{\big( {{{\widetilde L}_{j,\lambda (x,k)}}} \big)}_{{y_i}}}(y)} \big|\lesssim {2^{jn\sigma  + j\sigma }}, \,\text{for all } y\in\mathbb R^n.
\end{equation}
Thus, we also have
\begin{equation}\label{napGy2}
\big| {{{\big( {{\overline{{\widetilde L}}_{j,\lambda (z,k)}}} \big)}_{{t_i}}}( - u + hy)} \big| \lesssim {2^{jn\sigma  + j\sigma }}, \,\text{for all } y\in\mathbb R^n.
\end{equation}
Hence, by (\ref{napGy}), (\ref{napGy1}) and (\ref{napGy2}), we obtain that
\begin{equation}\label{napGyi}
\big| {{G_{{y_i}}}(y)} \big|\lesssim {2^{2jn\sigma  + j\sigma }}, \,\text{for all } y\in\mathbb R^n, i=1,2,...,n.
\end{equation}
From (\ref{paper2'}) and (\ref{napGyi}), the proof of (\ref{paper2}) is completed.
\\
\\
 $\triangleright$ \textit{Case 1.1}: $0 < h \le h_0 < 1$, here $h_0$ is small positive number to be determined. We have
\begin{align}
&{P_{{A_{j,\lambda (x,k)}}\circ \lambda(x,k)}}(y) - {P_{{A_{j,\lambda (z,k)}}\circ\lambda(z,k)}}( - u + hy)\nonumber
\\
&=\sum\limits_{2 \le |\alpha | \le d} {\left\{ {{{\big( {{A_{j,\lambda (x,k)}}\circ\lambda(x,k)} \big)}_\alpha } +O\big( h{{\left| {{A_{j,\lambda (z,k)}}\circ\lambda(z,k)} \right|}}\big)} \right\}y^\alpha}\nonumber
\\
&- h\sum\limits_{l= 1}^n {P_{{A_{j,\lambda (z,k)}} \circ \lambda (z,k)}^{(l)}} (u).{y_l} -{P_{{A_{j,\lambda (z,k)}}\circ\lambda(z,k)}}(- u)\nonumber
\\
&={Q_1}(y) + {Q_2}(y) - {P_{{A_{j,\lambda (z,k)}}\circ \lambda(z,k)}}(- u).\nonumber
\end{align}
We observe that $Q_2$ is terms of degree 1 in $y$ in the phase ${P_{{A_{j,\lambda (x,k)}}\circ \lambda(x,k)}}(y) - {P_{{A_{j,\lambda (z,k)}}\circ\lambda(z,k)}}( - u + hy)$, with
\[
P_{{A_{j,\lambda (z,k)}} \circ \lambda (z,k)}^{(l)}(u) = \sum\limits_{2 \le \left| \alpha  \right| \le d} {{\alpha _l}} {\big( {{A_{j,\lambda (z,k)}} \circ \lambda (z,k)} \big)_\alpha }u^{\alpha - e_l}.
\]
Hence, we have
\begin{equation}\label{phi}
\begin{split}
&\big| {{\Phi ^{\lambda (x,k)}} * {\widetilde \Phi ^{\lambda (z,k)}}\big({2^j}u/N(\lambda (z,k))\big)} \big|
\\
&\,\,\,\,\,\,\,\,\,\,\,\,\,\,\,\,\,\,\,\,\,\,\,\,\,\,\,\,\,\,\,\,\,\,\,\,\,\,\,\,\,\,\,\,= {\left( {{2^{ - j}}N(\lambda (z,k))} \right)^n}\Big| {\int\limits_{|y| \le \frac{{33}}{{32}}} {{e^{i\left\{ {{Q_1}(y) + {Q_2}(y)} \right\}}}} G(y)dy} \Big|.
\end{split}
\end{equation}
By the part $\rm (i)$ of Lemma \ref{mdeVDC}, there exists a positive constant $C >0$ independent of ${Q_1}, {Q_2}, G$ such that
\[
\Big| {\int\limits_{|y| \le \frac{{33}}{{32}}} {{e^{i\left\{ {{Q_1}(y) + {Q_2}(y)} \right\}}}} G(y)dy} \Big| \le {C}{\big| \lambda\big|^{ - 1/d}}\mathop {\sup }\limits_{y \in \big\{t\in \mathbb R^n: |t| \le \frac{{33}}{{32}}\big\} } \Big( {\big|G(y)\big| + \big| {\nabla G(y)} \big|} \Big),
\]
where let $|\lambda | = \sum\limits_{1 \le |\alpha | \le d} {|{\lambda _\alpha }} |$, with real coefficients ${\lambda _\alpha }$ of the polynomial function  ${Q_1}(y) + {Q_2}(y)$. Using a similar argument as in \cite{SteinWainger2001}, we also have $h_0\in (0,1)$ to $r\lesssim |\lambda|$. Thus, by (\ref{paper2}), it implies that
\begin{equation}\label{paper3}
\begin{split}
&\big| {{\Phi ^{\lambda (x,k)}} * {\widetilde \Phi ^{\lambda (z,k)}}\big({2^j}u/N(\lambda (z,k))\big)} \big| 
\\
&\,\,\,\,\,\,\,\,\,\,\,\,\,\,\,\,\,\,\,\,\,\,\,\,\,\,\,\,\,\,\,\,\,\,\,\,\,\,\,\,\,\,\,\,\,\,\,\,\,\,\,\,\,\,\,\,\,\,\,\,\,\,\,\,\,\,\,\,\,\,\,\,\lesssim {\big( {{2^{ - j}}N(\lambda (z,k))} \big)^n}{r^{-1/d}}{2^{2jn\sigma  + j\sigma }}{\chi _{{B_{33/16}}}}(u).
\end{split}
\end{equation}
\\
$\triangleright$ \textit{Case 1.2}: $h_0 < h \le 1$.
From (\ref{phi}), we get
\[
\big| {{\Phi ^{\lambda (x,k)}} * {\widetilde \Phi ^{\lambda (z,k)}}\big({2^j}u/N(\lambda (z,k))\big)} \big|
\lesssim {\left( {{2^{ - j}}N(\lambda (z,k))} \right)^n} \mathop {\sup }\limits_{y \in \big\{t\in \mathbb R^n: |t| \le \frac{{33}}{{32}}\big\} }{\big|G(y)\big|}.
\]
Hence, by (\ref{paper2}), we obtain that
\begin{equation}\label{paper5}
\big| {{\Phi ^{\lambda (x,k)}} * {\widetilde \Phi ^{\lambda (z,k)}}\big({2^j}u/N(\lambda (z,k))\big)} \big| \lesssim {\left( {{2^{ - j}}N(\lambda (z,k))} \right)^n} {2^{2jn\sigma+j\sigma }}.
\end{equation}
 For $\rho >0$, denote
 \[
 E_{{\lambda (z,k)}}^j = \left\{ {u \in {B_{32/16}}:\sum\limits_{l = 1}^n {\big| {P_{{A_{j,\lambda (z,k)}\circ\lambda(z,k)}}^{(l)}(u)} \big|}  \le \rho } \right\}.
 \]
Thus, it is not difficult to show that
\begin{equation}\label{paper4}
\begin{split}
&\big|{{\Phi ^{\lambda (x,k)}} * {\widetilde \Phi ^{\lambda (z,k)}}\big({2^j}u/N(\lambda (z,k))\big)} \big|
\\
&\,\,\,\,\,\,\,\,\,\,\,\,\,\,\,\,\,\,\,\lesssim {\big( {{2^{ - j}}N(\lambda (z,k))} \big)^n}{\rho ^{ - 1/d}}{2^{2jn\sigma  + j\sigma }}{\chi _{{B_{33/16}}\backslash E_{\lambda (z,k)}^j}}(u).
\end{split}
\end{equation}
We also denote $\widetilde E_{\lambda (z,k)}^j = \Big\{ {u \in {B_{32/16}}:\big| {\sum\limits_{l = 1}^n {P_{{A_{j,\lambda (z,k)}}\circ\lambda(z,k)}^{(l)}(u)} } \big| \le \rho } \Big\}$.
\\
We know that ${\sum\limits_{l = 1}^n {P_{{A_{j,\lambda (z,k)}}\circ\lambda(z,k)}^{(l)}(u)} }=\sum\limits_{l=1}^n \sum\limits_{2 \le \left| \alpha  \right| \le d} {{\alpha _l}} {\big( {{A_{j,\lambda (z,k)}} \circ \lambda (z,k)} \big)_\alpha }u^{\alpha - e_l} $ is a polynomial in $\mathbb R^n$ of degree $\leq d$. Thus, by the part $\rm(ii)$ of Lemma \ref{mdeVDC}, there exists a positve constant C such that
\[
\big| {\widetilde E_{\lambda (z,k)}^j} \big| \le C{\rho ^{1/d}}{\Big(\sum\limits_{l=1}^{n}{\sum\limits_{2 \le |\alpha | \le d} {{\alpha _l}{\big|{\big(A_{j,\lambda(z,k)}\circ \lambda(z,k)\big)}_\alpha}\big|} } \Big)^{ - 1/d}}, \,\text{for all}\,\rho>0.
\]
It is clear that $\sum\limits_{l=1}^{n}{\sum\limits_{2 \le |\alpha | \le d} {{\alpha _l}{\big|{\big(A_{j,\lambda(z,k)}\circ \lambda(z,k)\big)}_\alpha}\big|} }\geq r$. Therefore, by choosing $\delta  = \frac{1}{{6d}}$ and $\rho  = {\left( {\overline{c}} \right)^{ - d}}{r^{1/3}}$, with $\overline{c}$ being appropriately small, we estimate $\big| {\widetilde E_{\lambda (z,k)}^j} \big|\leq r ^{-4\delta}$. Since $E_{\lambda (z,k)}^j\subset \widetilde E_{\lambda (z,k)}^j$, we imply $\big| {E_{\lambda (z,k)}^j} \big| \leq r ^{-4\delta}$.
\\
\\
From (\ref{paper3}), (\ref{paper5}) and (\ref{paper4}), with any positive number $r$, we conclude that
\begin{equation}\label{paper6}
\begin{split}
&\big| {{\Phi ^{\lambda (x,k)}} * {\widetilde \Phi ^{\lambda (z,k)}}\big({2^j}u/N(\lambda (z,k))
\big)} \big|
\lesssim {\big( {{2^{ - j}}N(\lambda (z,k))} \big)^n}.{2^{2jn\sigma+j\sigma }}\times 
\\
&\,\,\,\,\,\,\,\,\,\,\,\,\,\,\,\,\,\,\,\,\,\,\,\,\,\,\,\,\,\,\,\,\,\times \left({r^{ - 6\delta }}{\chi _{{B_{33/16}}}}(u)
 + {r^{ - 2\delta }}{\chi _{{B_{33/16}}}}(u) + {\chi _{E_{\lambda (z,k)}^j}}(u)
\right). 
\end{split}
\end{equation}
\\
%%%%%%%%%%%%%%%%%%%%%%%%%%%%%%%%%%%%%%%%
\textit{$\circ$ Case 2}: If $h>1$ and any positive number $r$, by a similar argument as above,  we can also prove that
 \begin{equation}\label{paper7}
\begin{split}
& \big| {{\Phi ^{\lambda (z,k)}} * {\widetilde \Phi ^{\lambda (x,k)}}\big({2^j}u/N(\lambda (x,k))\big)} \big|
 \lesssim {\big( {{2^{ - j}}N(\lambda (x,k))} \big)^n}{2^{2jn\sigma+j\sigma }}\times
\\
&\,\,\,\,\,\,\,\,\,\,\,\,\,\,\,\,\,\,\,\,\,\,\,\,\,\,\,\,\,\,\,\,\,\times
\left( 
{r^{ - 6\delta }}{\chi _{{B_{33/16}}}}(u)
 + {r^{ - 2\delta }}{\chi _{{B_{33/16}}}}(u) + {\chi _{E_{\lambda (x,k)}^j}}(u)
 \right).
\end{split}
\end{equation}
Since  the results in (\ref{paper6}) and (\ref{paper7}), for all $r>0$, we conclude that 
\begin{align}\label{paper8}
\big| {{\kappa _{j,r,k}}(x,z)} \big| &\lesssim {\big( {{2^{ - j}}N(\lambda (z,k))} \big)^n}{2^{2jn\sigma+j\sigma }}\times
\nonumber
\\
& \,\,\,\,\,\,\,\,\,\,\,\,\,\,\,\,\,\,\,\times\left( \begin{array}{l}
{r^{ - 6\delta }}{\chi _{{B_{33/16}}}}\big({2^{ - j}}N(\lambda (z,k)).(x - z)\big ) 
\\
+\,{r^{ - 2\delta }}{\chi _{{B_{33/16}}}}\big({2^{ - j}}N(\lambda (z,k)).(x - z)\big)
\\
+\, {\chi _{E_{\lambda (z,k)}^j}}\big({2^{ - j}}N(\lambda (z,k)).(x - z)\big )
\end{array} \right)
\nonumber
\\
&+{\left( {{2^{ - j}}N(\lambda (x,k))} \right)^n}{2^{2jn\sigma+j\sigma }}\times
\\
\nonumber
&\,\,\,\,\,\,\,\,\,\,\,\,\,\,\,\,\,\,\,\times\left( \begin{array}{l}
{r^{- 6\delta}}{\chi _{{B_{33/16}}}}\big({2^{ - j}}N(\lambda (x,k)).(z - x)\big)
\\
 +\,{r^{ - 2\delta }}{\chi _{{B_{33/16}}}}\big({2^{ - j}}N(\lambda (x,k)).(z - x)\big)
\\
 + {\chi _{E_{\lambda (x,k)}^j}}\big({2^{ - j}}N(\lambda (x,k)).(z - x)\big)
\end{array} \right).
\\
\nonumber
\end{align}
Here, we condition that  $0 < r \le \big| {{A_{j,\lambda (x,k)}}\circ\lambda(x,k)} \big|,\big| {{A_{j,\lambda (z,k)}}\circ\lambda(z,k)} \big| < 2r$ and $E_{\lambda (z,k)}^j,E_{\lambda (x,k)}^j \subset {B_{33/16}}$ satisfy $\big| {E_{\lambda (z,k)}^j} \big|,\big| {E_{\lambda (x,k)}^j} \big| \leq {r^{- 4\delta }}$ and $\delta  = \frac{1}{{6d}}$.
\\
\\
%%%%%%%%%%%%%%%%%%%%%%%%%%%%%%%%%%%%%%%%%%%%%%%%%%%%%%%%%%%%%%%%%%%%%%%
%%%%%%%%%%%%%%%%%%%%%%%%%%%%%%%%%%%%%%%%%%%%%%%%%%%%%%%%%%%%%%%%%%%%%%%
$\bullet$ \textit{Step 2.2.2 :} The boundedness of $\mathfrak{T}^j _{\lambda (\cdot)}$ on $L^{2}(\ell^2,dx)$.\\
Thus, by (\ref{paper8}), we have
\\
$
\Big| {\left\langle {\left( {\mathfrak{T} ^{j,r,k}} \right){{\left( {\mathfrak{T}^{j,r,k}} \right)}^*}(f),g} \right\rangle } \Big|
$
\\
\\
$
 \lesssim {r^{ - 6\delta }}{2^{2jn\sigma  + j\sigma }}\int\limits_{{\mathbb R^n}} {\left| {f(z)} \right|} {\big( {{2^{ - j}}N(\lambda (z,k))} \big)^n}\left( {\int\limits_{|x - z| \le \frac{{{{33.2}^j}}}{{16.N(\lambda (z,k))}}} {\left| {g(x)} \right|dx} } \right)dz
$
\\
 \\
$
+{r^{ - 2\delta }}{2^{2jn\sigma  + j\sigma }}\int\limits_{{\mathbb R^n}} {\left| {f(z)} \right|} {\big( {{2^{ - j}}N(\lambda (z,k))} \big)^n}\left(  {\int\limits_{|x - z| \le \frac{{{{33.2}^j}}}{{16.N(\lambda (z,k))}}} {\left| {g(x)} \right|dx} } \right)dz
$
\\
\\
$
+{2^{2jn\sigma+j\sigma }}\int\limits_{{\mathbb R^n}} {\left| {f(z)} \right|} {\big( {{2^{ - j}}N(\lambda (z,k))} \big)^n}\left({\int\limits_{{\mathbb R^n}} {{\chi _{E_{\lambda (z,k)}^j}}\big({2^{ - j}}N(\lambda (z,k)).(x - z) \big)\left| {g(x)} \right|dx} }\right)dz
$
\\
\\
$
+{r^{ - 6\delta }}{2^{2jn\sigma  + j\sigma }}\int\limits_{{\mathbb R^n}} {\left| {g(x)} \right|} {\big( {{2^{ - j}}N(\lambda (x,k))} \big)^n}\left( {\int\limits_{|z - x| \le \frac{{{{33.2}^j}}}{{16.N(\lambda (x,k))}}} {\left| {f(z)} \right|dz} } \right)dx
$
\\
\\
$
+{r^{ - 2\delta }}{2^{2jn\sigma  + j\sigma }}\int\limits_{{\mathbb R^n}} {\left| {g(x)} \right|} {\big( {{2^{ - j}}N(\lambda (x,k))} \big)^n}\left( {\int\limits_{|z - x| \le \frac{{{{33.2}^j}}}{{16.N(\lambda (x,k))}}} {\left| {f(z)} \right|dz} } \right)dx
$
\\
\\
$
+{2^{2jn\sigma+j\sigma }}\int\limits_{{\mathbb R^n}} {\left| {g(x)} \right|} {\big( {{2^{ - j}}N(\lambda (x,k))} \big)^n}\left( {\int\limits_{{\mathbb R^n}} {{\chi _{E_{\lambda (x,k)}^j}}\big({2^{ - j}}N(\lambda (x,k)).(z - x)\big)\left| {f(z)} \right|dz} } \right)dx
$
\\
\\
$
= \sum\limits_{i = 1}^6 {{I_i}}.
$
\\
Applying the boundedness of the standard maximal function and Holder's inequality, we get
\[
I_1,I_4\lesssim {r^{ - 6\delta }}{2^{2jn\sigma  + j\sigma }}{\left\| f \right\|_{{L^2}}}{\left\| g \right\|_{{L^2}}},
\] 
and
\[
 I_2,I_5\lesssim {r^{ - 2\delta }}{2^{2jn\sigma  + j\sigma }}{\left\| f \right\|_{{L^2}}}{\left\| g \right\|_{{L^2}}}.
\] 
On the other hand, using Lemma \ref{Steinmde31}, we have
 $I_3, I_6\lesssim {2^{2jn\sigma+j\sigma }}{r^{ - 2\delta }}{\left\| f \right\|_{{L^2}}}{\left\| g \right\|_{{L^2}}}.$
 Therefore,
\begin{equation}\label{paper9}
\left\| {\mathfrak{T}^{j,r,k}} \right\|_{{L^2} \to {L^2}}^2 \lesssim 2^{2jn\sigma  + j\sigma }\left( {{r^{ - 6\delta }}{} + {r^{ - 2\delta }}} \right).
\end{equation}
Next, for $k \in\mathbb{N}, x\in\mathbb R^n$, we imply $N\big(A_{j,\lambda (x,k)}\circ \lambda(x,k)\big) = {2^j}$ and exist $c_0>1$ such that $N(v) \le {c_0}\left| v \right|,{\rm{ }}$ for all $v$ satisfying  $N(v) \ge 1$. A consequence of the above arguments is that 
\[
 \mathbb{R}^n = \bigcup\limits_{\ell = 0}^\infty  {U_{j,\,2^{j+\ell}/{c_0},\, k}},
\]
where ${U_{j,\,{2^{j+\ell}}/{c_0},\,k}}$'s have disjoint. Thus, by (\ref{paper9}), we estimate that
\begin{equation}\label{paper11}
{\big\|{\mathfrak{T} _{\lambda (\cdot)}^j\big(\vec{f}\big)} \big\|_{{L^{2}}(\ell^2, dx)}} \lesssim {2^{jn\sigma  + j\sigma /2 - j\delta }}{\big\| {\vec{f}}\big\|_{{L^{2}(\ell^2, dx)}}}.
\end{equation}
\\
%%%%%%%%%%%%%%%%%%%%%%%%%%%%%%%%%%%%%%%%%%%%%%%%%%%%%
%%%%%%%%%%%%%%%%%%%%%%%%%%%%%%%%%%%%%%%%%%%%%%%%%%%%%
$\bullet$ \textit{Step 2.2.3:} Applying interpolation estimates of $\mathfrak{T}^j _{\lambda (\cdot)}$.
\\
Let $\theta =\frac{1}{2}\left| \frac{2}{p} -1\right|+\frac{1}{2}\in (0,1)$ for convenience. We have $\left|\frac{2}{p}-1\right|<\theta$. It is clear that
\[\left\{ \begin{array}{l}
p(1+\theta)-2>0,
\\
2-p(1-\theta)>0.
\end{array} \right.
\]
We denote $s=\frac{2p\theta}{2-p(1-\theta)}$. From the above inequality, $s$ is defined well and $s>1$. Hence, by (\ref{paper*}) and Theorem \ref{theorem27}, we have
\begin{equation}\label{paper12}
{\big\|{\mathfrak{T} _{\lambda (\cdot)}^j\big(\vec{f}\big)} \big\|_{{L^{s}}(\ell^2, dx)}} \lesssim {2^{jn\sigma}}{\big\| {\vec{f}}\big\|_{{L^{s}(\ell^2, dx)}}}.
\end{equation}
By the inequalities (\ref{paper11}), (\ref{paper12}) and using Theorem \ref{interpolation2}, it is not difficult to show that
\begin{equation}\label{paper13}
\big\| { {\mathfrak{T}^j_{\lambda (\cdot)}\big(\vec{f}\big)}} \big\|_{L^{p}(\ell^2,dx)} \lesssim {2^{jn\sigma  + j(1-\theta)\sigma /2 - j\delta(1-\theta)}}\big\| {\vec{f}} \big\|_{L^{p}(\ell^2,dx)}.
\end{equation}
On the other hand, by assuming $\omega \in {A_p}$, there is an $\varepsilon >0$ such that ${\omega^{1 + \varepsilon }} \in {A_p}$. Using (\ref{paper*}) and Theorem \ref{theorem27}, one has
\begin{equation}\label{paper14}
\big\| { {\mathfrak{T}^j_{\lambda (\cdot)}(\vec{f} )}} \big\|_{L^{p}(\ell^2,\omega^{1 + \varepsilon })} \lesssim {2^{jn\sigma }}\big\| {\vec{f} } \big\|_{L^{p}(\ell^2,\omega^{1 + \varepsilon })}.
\end{equation}
From the inequalities (\ref{paper13}), (\ref{paper14}) and applying Corollary \ref{HquaSteinWeiss} for $\theta_1 =\frac{\varepsilon}{1+ \varepsilon}\in (0,1)$, we imply
\begin{equation}\label{paper15}
\big\| { {\mathfrak{T}^j_{\lambda (\cdot)}\big(\vec{f}\big)}} \big\|_{L^{p}(\ell^p,\omega)} \lesssim 
{2^{jn\sigma  + j(1-\theta)\theta_1.\sigma /2 - j\delta(1-\theta)\theta_1}}\big\| {\vec{f}} \big\|_{L^{p}(\ell^p,\omega)}.
\end{equation}
Now, we choose
\[\left\{ \begin{array}{l}
m = \frac{{2r}}{{r + 1}}\,\,\,\text{and}\,\,\,\theta_2 = \frac{2(p-r)}{(r+1)p-2r}, \,\text{if}\,\, p > r,
\\
\\
m = r \,\,\,\text{and}\,\,\,\theta_2 = \frac{1}{2}, \,\text{if}\,\, p = r, 
\\
\\
m = 2r \,\,\,\text{and}\,\,\,\theta_2 = \frac{2(r-p)}{2r-p}, \,\text{if}\,\, p < r.
\end{array} \right.
\]
Therefore, we have $\theta_2\in (0,1)$, $m>1$ and satisfy $\frac{1}{r}=\frac{\theta_2}{m}+\frac{1-\theta_2}{p}$.
Thus, by (\ref{paper*}) and Theorem \ref{theorem27} with $\omega\in A_p$, we get
\begin{equation}\label{paper16}
{\big\| {\mathfrak{T}^j_{\lambda (\cdot)}\big(\vec{f}\big)}\big\|_{{L^{p}(\ell^m,\omega)}}} \lesssim {2^{jn\sigma }}{\big\|{ {\vec{f} }} \big\|_{{L^{p}(\ell^m,\omega)}}}.
\end{equation}
Then, by (\ref{paper15}), (\ref{paper16}) and Lemma 2.2 of paper \cite{Andersen1981}, we have the following inequality 
\begin{equation}\label{paper17}
\big\| { {\mathfrak{T}^j_{\lambda (\cdot)}\big(\vec{f}\big)}} \big\|_{L^{p}(\ell^r,\omega)} \lesssim 
{2^{jn\sigma  + j(1-\theta)\theta_1(1-\theta_2)\sigma /2 - j\delta(1-\theta)\theta_1(1-\theta_2)}}\big\| {\vec{f} } \big\|_{L^{p}(\ell^r,\omega)}.
\end{equation}
We choose $\sigma = \text{min}\, \Big\{ {\dfrac{{\delta (1 - \theta)\theta_1(1 - {\theta _2})}}{{2n + (1 - \theta )\theta_1(1 - {\theta _2})}},\,\dfrac{1}{2}} \Big\}$ such that $\sigma  \in (0,1)$. Let us  $\beta_1  =  - \big\{n\sigma + (1-\theta)\theta_1(1-\theta_2)\sigma/2-\delta(1-\theta)\theta_1(1-\theta_2)\big\}$ be a positive real number for simple symbol. We will obtain the boundedness of the operators $\mathfrak{T}^j_{\lambda (\cdot)}$ on $L^{p}(\ell^r, \omega)$ space, i.e, 
\begin{equation}\label{paper19}
\big\| \mathfrak{T}^j_{\lambda (\cdot)}\big(\vec{f}\big )\big\|_{{L^{p}}(\ell^r,\omega)}
 \lesssim {2^{- j\beta_1 }}{\big\|{\vec{f} }}\big\|_{{L^{p}}(\ell^r,\omega)}.
\end{equation}
From (\ref{paper21}) and (\ref{paper19}), we conclude that
\\
\begin{equation}\label{Tj}
\big\|\mathcal{T}^j_{\lambda(\cdot)}\big(\vec{f} \big)\big\|_{{L^{p}}(\ell^r,\omega)}
\lesssim  \Big(2^{- j\beta_1}  + 2^{- j\sigma} +  \mu (2^{- j\sigma - 2})\Big) {\big\|\vec{f}\big\|}_{L^{p}(\ell^r,\omega)}.
\end{equation}
\\
%%%%%%%%%%%%%%%%%%%%%%%%%%%%%%%%%%%%%%%%%%%%%%%%%%%%%%%%%%%%%%%%%%%%%%%%%%%%%%%%%%%%%%%%%
%%%%%%%%%%%%%%%%%%%%%%%%%%%%%%%%%%%%%%%%%%%%%%%%%%%%%%%%%%%%%%%%%%%%%%%%%%%%%%%%%%%%%%%%%
$\bullet$ \textit{Step 3:} The estimate of $\mathcal{T}_{\lambda (\cdot)}.$\\
Thus, by (\ref{eq38}), (\ref{tau0}) and (\ref{Tj}), in order to  prove Theorem \ref{lemma31}, it is sufficient to show that 
\[
\sum\limits_{j=1}^{\infty} \mu (2^{- j\sigma - 2})<\infty.
\]
Indeed, we let $\varphi  \in (0,1)$, where $\varphi$ will be chosen later. It follows that
\[
\sum\limits_{j = 1}^\infty  {\mu ({2^{ - j\sigma  - 2}})}\leq {\big( {\left( {1 - \varphi } \right)\ln 2} \big)^{ - 1}}\sum\limits_{j = 1}^\infty  {\int\limits_{{2^{ - j\sigma  - 2}}}^{{2^{ - j\sigma  - (1 + \varphi )}}} {\frac{{\mu (t)}}{t}} dt}.
\]
Let us denote that ${S_m} = \sum\limits_{j = 1}^m {\int\limits_{{2^{ - j\sigma  - 2}}}^{{2^{ - j\sigma  - (1 + \varphi )}}} {\frac{{\mu (t)}}{t}dt} }$. We choose ${n_0} \in\mathbb{N}$ such that $n_0 > \frac{{1 - \sigma }}{\sigma }$ and take $\varphi  = \left( {1 + \frac{1}{{{n_0}}}} \right)(1 - \sigma )\in(0,1)$. Since $\varphi > 1-\sigma$, we have
\[
2^{-m\sigma-2}<2^{-m\sigma-(1+\varphi)}<...<2^{-j\sigma-2}<2^{-j\sigma-(1+\varphi)}<...<2^{-\sigma-2}<2^{-\sigma-(1+\varphi)}<1.
\]
Therefore,
\[
{S_m} \le \int\limits_{{2^{ - m\sigma  - 2}}}^1 {\frac{{\mu (t)}}{t}} dt.
\]
Hence, by (\ref{Dini}), we estimate
 \[
 \sum\limits_{j = 1}^\infty  {\mu ({2^{ - j\sigma  - 2}})}  \lesssim \int\limits_0^1 {\frac{{\mu (t)}}{t}dt}  < \infty,
\]
which completes the proof of Theorem \ref{lemma31}  for the case $\vec{f}\in S$. Finally, since $S$  is dense in $L^{p}(\ell^r,\omega)$, we may also extend the result to the whole of  the space $L^{p}(\ell^r,\omega)$.
\end{proof}
%%%%%%%%%%%%%%%%%%%%%%%%%%%%%%%%%%%%%%%
\begin{proof}[The proof of Theorem \ref{Theorem31}]
From Lemma \ref{Grafakos9.2.6}, there exists a real number $p_0\in (1,p)$ so that $w\in A_{p_0}$. Thus, by Theorem \ref{lemma31}, we have
\begin{equation}\label{Carlesonp0}
{\big\| {{\mathcal{T^*}  }\big( {\vec{f}} \big)} \big\|_{{L^{p_0}}({\ell^r}, \omega)}} \leq C_0 {\big\| {\vec{f}} \big\|_{{L^{p_0}}({\ell^r,} \omega)}},
\end{equation}
for all $\vec{f}\in L^{p_0}(\ell^r,\omega)$.
By choosing $p_1>p$, we imply $\omega\in A_{p_1}$. Thus, using Theorem \ref{lemma31} again, we aslo have
\begin{equation}\label{Carlesonp1}
{\big\| {{\mathcal{T^*}  }\big( {\vec{f}} \big)} \big\|_{{L^{p_1}}({\ell^r}, \omega)}} \leq C_1 {\big\| {\vec{f}} \big\|_{{L^{p_1}}({\ell^r,} \omega)}},
\end{equation}
for all $\vec{f}\in L^{p_1}(\ell^r,\omega)$.
From (\ref{Carlesonp0}) and (\ref{Carlesonp1}), applying Theorem \ref{interpolation1}, we finish the proof.
\end{proof}
%%%%%%%%%%%%%%%%%%%%%%%%%%%%%%%%%%%%%%%
Remark that, using Theorem \ref{interpolation1} for the scalar-valued functions, we also have the $L^{p, q}$ boundedness for the maximal Carleson type operator $\mathcal{T^*}$ in paper \cite{DingLiu2011} as follows.
\begin{theorem}\label{Theorem31''}
Suppose that $1<q\leq \infty$, $1\leq q'<p<\infty$, $\omega\in A_{p/q'}$ and $K(x)=\frac{\Omega(x)}{| x |^n}$, where $\Omega$ satisfies (\ref{kernelR12})-(\ref{kernelR3}). Then, we have 
\[
{\big\| {{\mathcal{T^*}  }\big( {{f}} \big)} \big\|_{{L^{p,q}}(\omega)}} \leq C {\big\| {{f}} \big\|_{{L^{p,q}}(\omega)}},
\]
for all ${f}\in L^{p,q}(\omega)$.
\end{theorem}
\begin{proof}[The proof of Theorem \ref{Theorem31'}] In step 2.2.3 of Theorem \ref{lemma31}, we need not to apply Lemma 2.2 of paper  \cite{Andersen1981}. By (\ref{paper13}), we have
\begin{equation}\label{paper13'}
\big\| { {\mathfrak{T}^j_{\lambda (\cdot)}\big({f}\big)}} \big\|_{L^{p}(dx)} \lesssim {2^{jn\sigma  + j(1-\theta)\sigma /2 - j\delta(1-\theta)}}\big\| {{f}} \big\|_{L^{p}(dx)}.
\end{equation}
On the other hand, by assuming ${\omega^{1 + \varepsilon }} \in {A_p}$, using (\ref{paper*}) and Theorem \ref{theorem27} for  the scalar-valued case, we get
\begin{equation}\label{paper14'}
\big\| { {\mathfrak{T}^j_{\lambda (\cdot)}({f} )}} \big\|_{L^{p}(\omega^{1 + \varepsilon })} \lesssim {2^{jn\sigma }}\big\| {{f} } \big\|_{L^{p}(\omega^{1 + \varepsilon })}.
\end{equation}
From the inequalities (\ref{paper13'}), (\ref{paper14'}) and applying Theorem \ref{noisuyStein} for $\theta_1 =\frac{\varepsilon}{1+ \varepsilon}\in (0,1)$, we imply
\[
\big\| { {\mathfrak{T}^j_{\lambda (\cdot)}\big({f}\big)}} \big\|_{L^{p}(\omega)} \lesssim 
{2^{jn\sigma  + j(1-\theta)\theta_1.\sigma /2 - j\delta(1-\theta)\theta_1}}\big\| {{f}} \big\|_{L^{p}(\omega)}.
\]
Therefore, we have
\begin{equation}\label{paper15'}
\big\| { {\mathfrak{T}^j_{\lambda (\cdot)}\big({f}\big)}} \big\|_{L^{p,\infty}(\omega)} \lesssim 
{2^{jn\sigma  + j(1-\theta)\theta_1.\sigma /2 - j\delta(1-\theta)\theta_1}}\big\| {{f}} \big\|_{L^{p,1}(\omega)}.
\end{equation}
By choosing $\sigma = \text{min}\, \Big\{ {\dfrac{{\delta (1 - \theta)\theta_1}}{{2n + (1 - \theta )\theta_1}},\,\dfrac{1}{2}} \Big\}$ and $\beta_1  =  - \big\{n\sigma + (1-\theta)\theta_1\sigma/2-\delta(1-\theta)\theta_1\big\}$, we will have the boundedness of the operators $\mathfrak{T}^j_{\lambda (\cdot)}$ on $L^{p,1}(\omega)$ space, i.e, 
\[
\big\| \mathfrak{T}^j_{\lambda (\cdot)}\big({f}\big )\big\|_{{L^{p,\infty}}(\omega)}
 \lesssim {2^{- j\beta_1 }}{\big\|{{f} }}\big\|_{{L^{p, 1}}(\omega)}.
\]
It is the biggest difference between Theorem \ref{lemma31} and Theorem \ref{Theorem31'} in the proof. The other results are estimated in the same way as Theorem \ref{lemma31}.  Therefore, by Theorem \ref{dinhly1}, Theorem \ref{theorem28'}, Corollary \ref{corollary29'} for the scalar-valued case and the density of $S$, we finish  the proof of Theorem \ref{Theorem31'}  for the whole of  the space $L^{p, 1}(\omega)$.
\end{proof}
%Obviously, Theorem \ref{Theorem31} allows us to imply the following interesting result.
%\begin{theorem}\label{Theorem32}
%Let $1<r<\infty$, $2\leq p <2q < \infty $, and $\omega\in A_p$. Then, we have 
%\begin{equation}
%{\big\| {{\mathcal{T^*}  }\big( {\vec{f}} \big)} \big\|_{{L^{p,q}}({\ell^r}, \omega)}} \leq C {\big\| {\vec{f}} \big\|_{{L^{p,q}}({\ell^r,} \omega)}},
%\end{equation}
%for all $\vec{f}\in L^{p,q}(\ell^r,\omega)$.
%\end{theorem}
%In particular, we can also give a similar argument as Theorem \ref{Theorem31} with $\eta = p$, except for using Theorem \ref{interpolation2} in the last section of step 2.2.3.  Thus, we deduce a beautiful result on $L^p(\ell^r,\omega)$ as follows.
%\begin{theorem}\label{Theorem33}
%If $1<p, r<\infty$, and $\omega\in A_p$,  then we have 
%\begin{equation}
%{\big\| {{\mathcal{T^*}  }\big( {\vec{f}} \big)} \big\|_{{L^{p}}({\ell^r}, \omega)}} \leq C {\big\| {\vec{f}} \big\|_{{L^{p}}({\ell^r,} \omega)}},
%\end{equation}
%for all $\vec{f}\in L^{p}(\ell^r,\omega)$.
%\end{theorem}
%%%%%%%%%%%%%%%%%%%%%%%%%%%%%%%%%
% Bibliography
%%%%%%%%%%%%%%%%%%%%%%%%%%%%%%%%%

\end{document}